\documentclass[11pt,reqno]{amsart}
\usepackage{amssymb}
\usepackage{amsmath}
\usepackage{amsthm}
\usepackage{graphicx}
\usepackage{caption}
\usepackage{bm}
\usepackage{subfigure}
\usepackage[colorlinks=true,urlcolor=blue,
citecolor=red,linkcolor=blue,linktocpage,pdfpagelabels,
bookmarksnumbered,bookmarksopen,backref=page]{hyperref}

\usepackage{cite} 

\usepackage[titletoc,title]{appendix}
\numberwithin{equation}{section} %riparte da zero ogni sezione
\newcounter{cont}[section] 

\newtheorem{thm}[cont]{Theorem}
\newtheorem{prop}[cont]{Proposition}
\newtheorem{lem}[cont]{Lemma}

\theoremstyle{definition}
\newtheorem{defn}[cont]{Definition}
 \theoremstyle{remark}
 \newtheorem{rem}[cont]{Remark}

\newcommand{\R}{\mathbb{R}}
\newcommand{\e}{\varepsilon}

\begin{document}

\title[Generalized Allen--Cahn model with degenerate diffusivity]{Long-time behavior of solutions to the generalized Allen--Cahn model with degenerate diffusivity}

\author[R. Folino]{Raffaele Folino}
\address[R. Folino]{Departamento de Matem\'aticas y Mec\'anica\\Instituto de Investigaciones en Matem\'aticas Aplicadas y en 
Sistemas\\Universidad Nacional Aut\'{o}noma de M\'{e}xico\\Circuito Escolar s/n, Ciudad Universitaria, C.P. 04510\\Cd. de M\'{e}xico (Mexico)}
\email{folino@mym.iimas.unam.mx}

\author[L. F. L\'{o}pez R\'{\i}os]{Luis F. L\'{o}pez R\'{\i}os}
\address[L. F. L\'{o}pez R\'{\i}os]{Departamento de Matem\'aticas y Mec\'anica\\Instituto de Investigaciones en Matem\'aticas Aplicadas y 
en Sistemas\\Universidad Nacional Aut\'{o}noma de M\'{e}xico\\Circuito Escolar s/n, Ciudad Universitaria, C.P. 04510\\Cd. de M\'{e}xico (Mexico)}
\email{luis.lopez@iimas.unam.mx}
	
\author[R. G. Plaza ]{Ram\'{o}n G. Plaza}
\address[R. G. Plaza]{Departamento de Matem\'aticas y Mec\'anica\\Instituto de Investigaciones en Matem\'aticas Aplicadas 
y en Sistemas\\Universidad Nacional Aut\'{o}noma de M\'{e}xico\\Circuito Escolar s/n, Ciudad Universitaria, C.P. 04510\\Cd. de M\'{e}xico (Mexico)}
\email{plaza@mym.iimas.unam.mx}
 
 \begin{abstract}
 The generalized Allen--Cahn equation,
 \[
 u_t=\e^2(D(u)u_x)_x-\frac{\e^2}2D'(u)u_x^2-F'(u),
 \]
 with nonlinear diffusion, $D = D(u)$, and potential, $F = F(u)$, of the form
 \[
 D(u) = |1-u^2|^{m}, \quad \text{or} \quad D(u) = |1-u|^{m}, \quad m >1,
 \]
 and
 \[
 F(u)=\frac{1}{2n}|1-u^2|^{n},  \qquad n\geq2,
 \]
 respectively, is studied. These choices correspond to a reaction function that can be derived from a double well potential, and to a generalized degenerate diffusivity coefficient depending on the density $u$ that vanishes at one or at the two wells, $u = \pm 1$. It is shown that interface layer solutions that are equal to $\pm 1$ except at a finite number of thin transitions of width $\e$ persist for an either exponentially or algebraically long time, depending upon the interplay between the exponents $n$ and $m$. For that purpose, energy bounds for a renormalized effective energy potential of Ginzburg--Landau type are derived.
 \end{abstract}
 
%\subjclass[2010]{35K20, 35K57, 35B36, 82B26}

\maketitle

%\tableofcontents

%\nocite{*}

\section{Introduction}
In this paper, we consider the equation
\begin{equation}\label{eq:model}
	u_t=\e^2(D(u)u_x)_x-\frac{\e^2}2D'(u)u_x^2-F'(u), \qquad x \in (a,b)\subset\R, \,\, t > 0,
\end{equation}
where $\e>0$ is a small parameter, the diffusion coefficient $D:\R\to\R$ is a nonnegative function that vanishes just in one or two points, and
$F:\R\to\R$ is a double well potential with wells of equal depth.
More precisely, we assume that $D$ and $F$ are given by either
\begin{equation}\label{eq:D-F}
	D(u):=|1-u^2|^{m}, \qquad m>1, \qquad \qquad F(u)=\frac{1}{2n}|1-u^2|^{n},  \qquad n\geq2,
\end{equation}
or 
\begin{equation}\label{eq:D-F-onezero}
	D(u):=|1-u|^{m}, \qquad m>1,  \qquad \qquad F(u)=\frac{1}{2n}|1-u^2|^{n}, \qquad n\geq2.
\end{equation}
We consider equation \eqref{eq:model} on a bounded interval $[a,b]$ of the real line with homogeneous Neumann boundary conditions of the form
\begin{equation}\label{eq:Neu}
	u_x(a,t)=u_x(b,t)=0, \qquad \qquad \forall \, t>0,
\end{equation}
and initial condition given by
\begin{equation}\label{eq:initial}
 	u(x,0)=u_0(x), \qquad \qquad x\in[a,b].
\end{equation}

Assuming that the solution $u$ is sufficiently smooth and expanding the first term in the right hand side of \eqref{eq:model}, we can rewrite it as
\begin{equation}\label{eq:model2}
	u_t=\e^2D(u)u_{xx}+\frac{\e^2}2D'(u)u_x^2-F'(u).
\end{equation}

\smallskip

Equation \eqref{eq:model} (equivalently, \eqref{eq:model2}) arises within the framework of one-dimensional phase separation theory in binary alloys by Allen and Cahn \cite{AlCa79} and by Cahn and Hilliard \cite{CaHi58,Cahn59}. It reduces to the classical Allen--Cahn equation in the case where the diffusivity coefficient is assumed to be constant, $D \equiv D_0 > 0$. In the classical theory, the scalar field $u$ (called \emph{order parameter} or simply \emph{density}) interpolates between two pure phase components of the alloy, $u = -1$ and $u = +1$,  and the parameter $\e > 0$ measures the interface width separation. In the limit when $\e \to 0^+$, solutions to \eqref{eq:model} describe sharp interface layers separating the phases. Under such circumstances, these layers are transient solutions that appear to be stable but which, in fact, destroy its phase separating structure after an exponentially long time. This behavior is known in the literature as \emph{metastability} (see \cite{BrKo90,CaPe89,FuHa89} and the references cited therein). In most works, consideration has been restricted to models with constant diffusivity coefficients. It is expected from the phenomenological and theoretical viewpoints, however, that diffusion (or mobility/stiffness) can depend on the order parameter $u$ (it is to be observed that a density-dependent mobility coefficient appears in the original derivation of the Cahn--Hilliard equation \cite{CaHi71}). Consequently, many extended Cahn--Hilliard or Allen--Cahn models that take these nonlinearities into account have been proposed in the literature (see, e.g., \cite{Cahn61,CaHi71,CEN-C96,TaCa94}) and which, in turn, have motivated their associated mathematical analyses (cf. \cite{CISc16,CISc13,DaiDu16a,DGN-C99,EllGa96,EllGa97,FHLP20}). An interesting mathematical feature appears when the nonlinearity in the diffusion coefficient is \emph{degenerate}, meaning that diffusion approaches zero when the density $u$ approaches one or two of the pure phases. This behavior occurs in some physical models as well. For example, in the description of some binary alloys the diffusivity seems to be zero outside a relatively narrow interfacial band (cf. \cite{CaTa94,TaCa94,EllGa96}), that is, $D$ is zero outside the grain boundary including, e.g., the pure phases $u = \pm 1$, and positive inside.  Actually, a function $D = D(\cdot)$ of the form in \eqref{eq:D-F} can be justified under thermodynamic considerations (see Taylor and Cahn \cite{TaCa94}). Among the new mathematical features one finds that some equations with degenerate diffusion possess finite speed of propagation of initial disturbances, in contrast with the strictly parabolic case (see, e.g., \cite{GiKe96}). Another feature is the emergence of traveling waves of \emph{sharp} \cite{Hos87,SaMa97} or \emph{compacton} type \cite{CISc16,CiSS19}. Model equations with degenerate diffusivities are not exclusive of the description of binary alloys, but also appear in studies of two-phase fluid flow in a capillary tube (cf. \cite{CuJu12}), which adopt the classical approach of phase separation by Cahn and Hilliard, implement a degenerate higher-order diffusivity coefficient, and describe compacton formation. To sum up, equations of the form \eqref{eq:model} with nonlinear functions given by \eqref{eq:D-F} or by \eqref{eq:D-F-onezero} encompass degenerate, density-dependent diffusion coefficients that vanish at one or two equilibrium points of the reaction (which has the form of the derivative of a double-well potential) and are of interest as models of phase separation in alloys and in other physical applications.

This paper studies the long-time behavior of solutions of interface layer type to equation \eqref{eq:model}, paying particular attention to the phenomenon of metastability.  Motivated by previous analyses on the topic for the classical Allen--Cahn equation (cf. \cite{BrKo90,CaPe89,CaPe90,FuHa89}), in this work we apply the energy approach of Bronsard and Kohn \cite{BrKo90} to rigorously prove the existence of metastable states for the initial boundary-value problem (IBVP) \eqref{eq:model}, \eqref{eq:Neu} and \eqref{eq:initial}. In a recent contribution \cite{FHLP20}, we studied the phenomenon of metastability of solutions to a related equation with density-dependent coefficients. The main differences between the model studied in \cite{FHLP20} and \eqref{eq:model} are (i) that equation \eqref{eq:model} has a different structure, being the $L^2$-gradient flow of a generalized Ginzburg--Landau energy functional (see \eqref{eq:energy} below) whereas the model in \cite{FHLP20} is in conservation form; and, (ii) that in the present case the diffusivity is allowed to be degenerate, vanishing at one or both of the pure phases. Moreover, this work emphasizes the interplay between the diffusion function, $D = D(\cdot)$, and the potential, $F = F(\cdot)$, exemplified by the relation between their (real) exponents $m>1$ and $n\geq 2$. For instance, it is proved that, in the case when $n = m+2$ and for any $m>1$, interface layer solutions exhibit an exponentially slow motion, that is, they maintain the same transition layer structure for (at least) an exponentially long time and the layers move with an exponentially small speed (see Theorems \ref{thm:main} - \ref{thm:interface} below). This result generalizes to degenerate diffusions the known classical results on the Allen--Cahn equation with $m=0$, $n=2$ (see, among other works, \cite{CaPe89,ChenX04}). In the case where $n > m+2$, our analysis predicts algebraic slow motion of interface layer solutions, showing that metastability (understood as exponentially slow motion of interface layers) is a consequence of the relation between the diffusion and the reaction, or to the degeneracy of one with respect to the other. When diffusion vanishes only at one point (like in \eqref{eq:D-F-onezero}) similar results, depending upon the interplay between reaction and diffusion, arise. By considering equations of the form \eqref{eq:model} with nonlinearities satisfying either \eqref{eq:D-F} or \eqref{eq:D-F-onezero} for any $m >1$ and any $n \geq 2$, our work encompasses a large class of models which can be of interest for the scientific community.

\begin{rem}
Our attention is focused on the explicit formulas \eqref{eq:D-F}-\eqref{eq:D-F-onezero} for the diffusivity $D$ and the potential $F$ 
only for the sake of simplicity and for the readability of the paper;
in fact, what is very important for our results is that the functions $D,F$ behaves as \eqref{eq:D-F}-\eqref{eq:D-F-onezero} in a neighborhood of $-1$ and $+1$.
For instance, one can extend the results contained in this paper to a generic diffusivity $D\in C^1(\R)$ satisfying
\begin{equation*}
\begin{aligned}
	&D(\pm1)=0, \qquad \qquad D(u)>0 \quad \forall\, u\neq\pm1,  \\
	&d_1 |1\pm u|^{m-1} \leq D'(u) \leq d_2 |1\pm u|^{m-1}, \qquad\qquad \mbox{ for } \qquad |u\pm1|<\eta,
	\end{aligned}
\end{equation*}
for some constants $0<d_1\leq d_2$, $\eta>0$ and $m>1$,  
and a generic potential $F\in C^2(\R)$ satisfying 
\begin{equation*}
\begin{aligned}
	&F(\pm1)=F'(\pm1)=0, \qquad \qquad F(u)>0 \quad \forall\, u\neq\pm1,  \\
	&\lambda_1 |1\pm u|^{n-2} \leq F''(u) \leq \lambda_2 |1\pm u|^{n-2}, \qquad\qquad \mbox{ for } \qquad |u\pm1|<\eta,
	\end{aligned}
\end{equation*}
for some constants $0<\lambda_1\leq\lambda_2$, $\eta>0$ and $n\geq2$.
\end{rem}

The rest of the paper is structured as follows. In Section \ref{sec:energyest} we establish an energy identity that plays a crucial role in the forthcoming analysis. First, we make precise the notion of strong solution of the IBVP (see Definition \ref{def:weak}) and then show the aforementioned energy identity (see Lemma \ref{lem:energy}) for any of these solutions with improved regularity. The existence of strong solutions to the IBVP is beyond the scope of this work. Section \ref{sec:stationary} is devoted to proving the existence of stationary solutions to equation \eqref{eq:model} with homogeneous Neumann boundary conditions \eqref{eq:Neu}. This is done for all values of $n$ and $m$ under consideration and yield smooth and sharp wave solutions depending on the relation between these two parameters. These stationary wave solutions constitute fundamental pieces to construct stationary transition layer solutions. In Section \ref{sec:compactons} we prove existence of some particular stationary solutions to \eqref{eq:model}-\eqref{eq:Neu} in the case $n<m+2$, known as \emph{compactons} (see the two examples in Figure \ref{fig:compactons} or more examples in \cite{CISc16}). Studying stability or attractiveness of the compactons is an interesting open problem; however, in this paper we are interested in proving existence of metastable or algebraically slowly moving patterns for the model \eqref{eq:model},
that is, to prove existence of non-stationary solutions which maintain a transition layer structure, oscillating between $-1$ and $+1$ (without touching them), for a very long time $T_\e$, which satisfies $\displaystyle\lim_{\e\to0^+}T_\e=+\infty$. Therefore, we mainly focus the attention on the case $n\geq m+2$, where compactons do not exist. 

In Section \ref{sec:lower} we prove that there exists a lower bound for the Ginzburg--Landau energy functional \eqref{eq:energy}, which is a key ingredient for the subsequent analysis. It is worth noticing that this result is purely variational in character and that equation \eqref{eq:model} plays no role. The lower bound depends upon the relation between $n$ and $m$ (see Propositions \ref{prop:lower} and \ref{prop:lower_deg} for the $n = m+2$ and $n > m+2$ cases, respectively). 
Moreover, the main novelty of the results of Section \ref{sec:lower} is that they concern the generalized Ginzburg--Landau functional \eqref{eq:energy},
with \emph{degenerate} functions $D,F$ as in \eqref{eq:D-F} or \eqref{eq:D-F-onezero};
in particular, thanks to the results contained in Section \ref{sec:lower} it is easy to study the $\Gamma$-convergence of the functional \eqref{eq:energy},
when $D,F$ are given by \eqref{eq:D-F} or \eqref{eq:D-F-onezero}, and to extend to such \emph{degenerate case} the results of \cite{OwSt91} (at least in the one-dimensional case),
for details see the discussion before Proposition \ref{prop:lower}.

The central section \ref{sec:slow} is devoted to establish the slow motion of the transition layer solutions, yielding the main Theorems \ref{thm:main}-\ref{thm:interface}. The latter establish the exponential vs. algebraic slow motions based on the relation between the parameter values $n$ and $m$, and on the choice of double \eqref{eq:D-F} or single degenerate \eqref{eq:D-F-onezero} diffusion coefficient. Some conclusive remarks can be found in Section \ref{sec:conclusions}.

\section{An energy identity}\label{sec:energyest}
In this section we formulate the notion of solution to equation \eqref{eq:model}, with homogeneous Neumann boundary conditions \eqref{eq:Neu} 
and initial condition \eqref{eq:initial}, and provide some important energy estimates for it.
We also consider \emph{general} diffusion coefficients and reaction terms satisfying the following assumptions.
\begin{itemize}
\item On the diffusion coefficient we assume
\begin{equation*}
    D \in C^1[-1,1], \quad D(-1)D(1)=0, \quad \text{and}     \quad D>0 \text{ in } (-1,1).
\end{equation*}
To avoid very weak diffusions, we assume also that at least one derivative of $D$ is different from zero at $-1$ and $1$.
\item On the reaction we suppose 
\begin{equation*}
	F \in C^1[-1,1], \quad\qquad F(\pm1)=F'(\pm1)=0 \quad \text{ and } \qquad F>0 \text{ in } (-1,1).
\end{equation*}
\end{itemize}

It is important to mention that the generalized Allen--Cahn equation \eqref{eq:model} is associated to the gradient flow 
of the generalized Ginzburg--Landau energy functional (see \cite[Appendix A]{CISc16})
\begin{equation}\label{eq:energy} 
	E_\e[u]:=\int_a^b\left[\frac\e2D(u)u_x^2+\frac{F(u)}{\e}\right]\,dx.
\end{equation}
This quantity will be very useful to estimate important norms of solutions to \eqref{eq:model}.

Consider the space-time rectangle $Q_T = (a,b)\times(0,T)$ with lateral boundary $\Sigma_T = \{ a,b \} \times (0,T)$; if $T=\infty$, we just write $Q$ and $\Sigma$. Let us also write equation (\ref{eq:model}) in the usual filtration form \cite{Vaz07}
\begin{equation}\label{eq:model_filtration}
  u_t = \e^2 \Psi_1(u)_{xx} - \frac{\e^2}{2} \Psi_2(u)_x^2 + \frac{\e^2}{2} \Psi_3(u)_x^2 -F'(u),
\end{equation}
where
\begin{itemize}
\item $\Psi_1' = D$;
\item $(\Psi_2')^2$ is the positive part of $D'$;
\item $(\Psi_3')^2$ is the negative part of $D'$;
\end{itemize}
and $\Psi_i(-1)=0$, $i=1,2,3$. For problem (\ref{eq:model_filtration}), with boundary and initial conditions \eqref{eq:Neu}-\eqref{eq:initial}, 
we consider the strong solution notion \cite{Vaz07}. 
This notion allows to compute all the derivatives appearing in the equation as functions rather than distributions and
in this way the equation is satisfied almost everywhere (a.e.) in its domain.
\begin{defn}\label{def:weak}
Let us consider a measurable function $u_0$, with $-1 \le u_0 \le 1$.  
We say that a measurable function $u$, with $-1 \le u \le 1$, is a \emph{strong solution} to \eqref{eq:model_filtration} with boundary conditions \eqref{eq:Neu} and initial
condition \eqref{eq:initial}, if
\begin{itemize}
\item[(i)] $$u_t \in L^2((0,T),L^2(a,b)), \qquad \Psi_1(u) \in L^2((0,T),H^2(a,b)),$$ and 
$$\Psi_2(u), \Psi_3(u) \in L^1_\textrm{loc}((0,T), W^{1,1}_\textrm{loc}(a,b));$$
\item[(ii)] equation (\ref{eq:model_filtration}) is satisfied a.e. in $Q_T$;
\item[(iii)] the boundary and initial conditions \eqref{eq:Neu}-\eqref{eq:initial} are satisfied in the trace sense:
\begin{equation*}
    T_0(u) = u_0 \quad \text{and} \quad T_\Sigma(\Psi_1(u)_x)=0,
\end{equation*}
where $T_0:H^1((0,T),L^2(a,b)) \to L^2(a,b)$ and $T_\Sigma: L^2((0,T),H^1(a,b)) \to L^2(\Sigma_T)$ are the trace operators on the parabolic boundary of the problem, 
see \cite{adams2003sobolev}.
\end{itemize}
\end{defn}

The main result of this section is the following.
\begin{lem}[energy identity]
\label{lem:energy}
Let $u \in H^1((0,T),W^{1,4}(a,b))$ be a strong solution to \eqref{eq:model_filtration}, with boundary conditions (\ref{eq:Neu}) and initial condition \eqref{eq:initial}, 
then $u$ satisfies the energy identity
\begin{equation}\label{eq:energyestimates}
	E_\e[u](0)-E_\e[u](T) = \e^{-1} \iint_{Q_T} u_t^2 \, dxdt.
\end{equation}
\end{lem}
\begin{proof}
Let us first multiply (\ref{eq:model_filtration}) by $u_t$ and integrate on the rectangle $Q_T$ to obtain
\begin{equation}\label{eq:model_int}
	\e^{-1} \iint_{Q_T}u_t^2 \, dxdt = - \iint_{Q_T} \left[ \e\Psi_1(u)_x u_{tx} + \frac{\e}2D'(u)u_x^2u_t + \frac{F'(u)}{\e} u_t \right]\, dxdt,
\end{equation}
where the first term in the right-hand side was obtained by integration by parts in space; observe that all the integrals are well defined, due to the hypothesis on $u$. 
Moreover, the right-hand side of the previous equation can be computed by approximating $u$ with smooth functions in the following way: 
since $u \in H^1((0,T),W^{1,4}(a,b))$, there exists a sequence $\{\varphi_n \}_{n \in \mathbb{N}} \subset C^\infty(\overline{Q_T})$ such that
\begin{gather}
    \varphi_n \to u, \quad \;\,\text{in } H^1(Q_T), \label{approx_1} \\
    \varphi_{nx}^2 \to u_x^2 \quad \text{in } L^2(Q_T), \label{approx_2}\\
    \varphi_{ntx} \to u_{tx} \quad \text{in } L^2(Q_T), \label{approx_3},
\end{gather}
as $n \to \infty$. 
On the other hand, observe that \eqref{approx_1} implies $\Psi_1(\varphi_n) \to \Psi_1(u)$ in $L^2(Q_T)$, and then 
$\Psi(\varphi_n)_x \rightharpoonup \Psi(u)_x$ weakly in $L^2(Q_T)$; 
this weakly convergence together with \eqref{approx_1}-\eqref{approx_3} imply that \eqref{eq:model_int} can be written as
\begin{equation}\label{eq:model_int_lim}
 	\begin{aligned}
      		\e^{-1} \iint_{Q_T}u_t^2 \, dxdt &= 
		- \lim_{n \to \infty} \iint_{Q_T} \left[ \e \Psi_1(\varphi_n)_x \varphi_{ntx} + \frac{\e}2D'(\varphi_n)\varphi_{nx}^2\varphi_{nt} + \frac{F'(\varphi_n)}{\e} \varphi_{nt} \right]\, dxdt \\
      		&= \lim_{n \to \infty} I_n.
    	\end{aligned}
\end{equation}
Thanks to the smoothness of the functions in the right-hand side of previous equation, the expressions inside can be calculated in the classical way: 
\begin{equation}\label{eq:I_n}
   	\begin{aligned}
      		I_n &= - \iint_{Q_T} \left[ \e D(\varphi_n) \varphi_{nx} \varphi_{nxt} + \frac{\e}2D'(\varphi_n)\varphi_{nx}^2\varphi_{nt} + \frac{F(\varphi_n)_t}{\e} \right]\, dxdt \\
      		&= - \iint_{Q_T} \left[ \frac{\e}{2} D(\varphi_n) (\varphi_{nx}^2)_t + \frac{\e}2D'(\varphi_n)\varphi_{nx}^2\varphi_{nt} + \frac{F(\varphi_n)_t}{\e} \right]\, dxdt \\
      		&= E_\e[\varphi_n](0) - E_\e[\varphi_n](T),
    	\end{aligned}
\end{equation}
where the last equation was obtained by integration by parts in time.

To finish the proof we have to take $n \to \infty$ in (\ref{eq:I_n}). 
To this end we use the theory of boundary traces \cite{adams2003sobolev}: for a fixed $s \in [0,T]$ there exists a continuous operator 
$$T_s:W^{1,1}((0,T),L^1(a,b)) \to L^1(a,b)$$ such that $T_s(\varphi) = \varphi(\cdot,s)$ for all $\varphi \in W^{1,1}((0,T),L^1(a,b)) \cap C(\overline{Q_T})$. 
On the other hand, due to \eqref{approx_1}-\eqref{approx_3}, we deduce that
\begin{align*}
   	(D(\varphi_n)\varphi_{nx}^2)_t &= D'(\varphi_n) \varphi_{nx}^2 \varphi_{nt} + 2D(\varphi_n) \varphi_{nx} \varphi_{nxt} \\
	&\to D'(u)u_x^2u_t + 2D(u)u_xu_{xt}
\end{align*}
in $L^1(Q_t)$. 
Therefore, we have $T_s(D(\varphi_n)\varphi_{nx}^2) \to T_s(D(u)u_x^2)$ in $L^1(a,b)$ and then
\begin{align*}
	E_\e[\varphi_n](s) &= \int_a^b\left[\frac\e2D(\varphi_n)\varphi_{nx}^2+\frac{F(\varphi_n)}{\e}\right]\,dx \\
       	&\to \int_a^b\left[\frac\e2D(u)u_x^2+\frac{F(u)}{\e}\right]\,dx = E_\e[u](s), \qquad \text{as } n \to \infty.
\end{align*}
This limit with $s=0$ and $T$, together with \eqref{eq:I_n} and \eqref{eq:model_int_lim} imply the energy estimate \eqref{eq:energyestimates},
and the proof is complete.
\end{proof}
The energy identity \eqref{eq:energyestimates} is crucial to study the slow motion of the solution to \eqref{eq:model}-\eqref{eq:Neu}-\eqref{eq:initial};
thus, in Section \ref{sec:slow} we will consider sufficiently smooth solutions such that \eqref{eq:energyestimates} holds true.
Let us stress that in Lemma \ref{lem:energy} we assume that the solution is more regular than Definition \ref{def:weak};
anyway, the \emph{weaker} Definition \ref{def:weak} will be helpful in Section \ref{sec:stationary},
where we will prove existence of particular stationary solutions that satisfy all the requirements of Definition \ref{def:weak}.

\section{Stationary solutions}
\label{sec:stationary}

In this section, we investigate some properties of the stationary solutions to the equation \eqref{eq:model} with homogeneous Neumann boundary conditions \eqref{eq:Neu},
when the diffusion coefficient $D$ and the potential $F$ are given in \eqref{eq:D-F} or \eqref{eq:D-F-onezero},
namely we discuss the properties of the solutions $\varphi:=\varphi(x)$ to the boundary value problem
\begin{equation}\label{eq:stationary}
	\e^2(D(\varphi)\varphi')'-\frac{\e^2}2D'(\varphi)(\varphi')^2=F'(\varphi), \qquad x\in(a,b), \qquad \qquad \varphi'(a)=\varphi'(b)=0.
\end{equation}
To do that, we focus the attention to some special solutions to \eqref{eq:model} in the whole real line.

\subsection{Standing waves}\label{sec:standing}
To start with, we discuss the existence of traveling waves for \eqref{eq:model} in the whole real line.
Traveling waves (or traveling fronts) are special solutions of the form $u(x,t)=\phi(x-ct)$,
where $c\in\R$ is the speed of the wave and $\phi:\R\to\R$ is the wave profile function. 
We are interested in monotone traveling wave solutions to \eqref{eq:model} connecting the two stable states $u=-1$ and $u=1$;
thus, we introduce the variable $\xi=x-ct$ and denote by $\phi:=\phi(\xi)=\phi(x-ct)$ and by $':=\frac{d}{d\xi}$.
Assuming that $u(x,t)=\phi(x-ct)=\phi(\xi)$ is an increasing solution to \eqref{eq:model2} (at least twice differentiable), we obtain
\begin{equation}\label{eq:travel}
	\begin{aligned}
		&-c\phi'=\e^2D(\phi)\phi''+\frac{\e^2}{2}D'(\phi)(\phi')^2-F'(\phi), \qquad  \xi\in(\omega_-,\omega_+),\\
		&\phi'(\xi)>0, \qquad  \xi\in(\omega_-,\omega_+), \qquad \phi(\omega_\pm)=\pm1, 	 \; \qquad D(\phi)(\phi')^2(\omega_\pm)=0,
	\end{aligned}
\end{equation}
for some $-\infty\leq\omega_-<\omega_+\leq+\infty$.
Let us stress that, in the case of non-degenerate diffusion $D>0$, one has $\omega_\pm=\pm\infty$, 
because the profile $\phi$ can not reach the states $u=\pm1$;
however, when the diffusion vanishes in one or both states to be connected, it is possible to have $\omega_\pm\in\R$.
Multiplying by $\phi'$ the first equation in \eqref{eq:travel}, we deduce
\begin{equation*}
	-c(\phi')^2=\left[\frac{\e^2}{2}D(\phi)(\phi')^2-F(\phi)\right]'.
\end{equation*}
By integrating in $(\omega_-,\omega_+)$ and using the second equation in \eqref{eq:travel}, we end up with
\begin{equation*}
	-c\int_{\omega_-}^{\omega_+}(\phi')^2\,d\xi=F(+1)-F(-1)=0.
\end{equation*}
Therefore, as in the case of constant diffusion, we can only have $c=0$ and there are no traveling waves with speed $c\neq0$ 
connecting the two global minimal points of the potential $F$.
Let us then focus the attention on the case $c=0$ and introduce the following special solution.
\begin{defn}
A function $\Phi:\R\to\R$ is called \emph{non-decreasing standing wave} to \eqref{eq:model} if it is a solution to the problem
\begin{equation}\label{eq:ST}
	\begin{cases}
		\displaystyle\e^2(D(\Phi)\Phi')'-\frac{\e^2}2D'(\Phi)(\Phi')^2-F'(\Phi)=0,  \qquad \qquad     x\in(\omega_1,\omega_2),\\
		\Phi(x)=-1, \qquad x\leq\omega_1,  \qquad\quad\qquad \Phi(x)=1, \qquad \;x\geq\omega_2,\\
		\displaystyle\lim_{x\to\omega_1}D(\Phi)(\Phi')^2(x)=\lim_{x\to\omega_2}D(\Phi)(\Phi')^2(x)=0, \quad \quad	\Phi'(x)>0, \quad  x\in(\omega_1,\omega_2), 
	\end{cases}
\end{equation}
\end{defn}
for some $-\infty\leq\omega_1<\omega_2\leq+\infty$, where the second condition has to be read as $\displaystyle\lim_{x\to-\infty}\Phi(x)=-1$ 
$\left(\displaystyle\lim_{x\to+\infty}\Phi(x)=+1\right)$ if $\omega_1=-\infty$ ($\omega_2=\infty$).

In the following proposition, we establish the existence of a non-decreasing standing wave and show that 
its behavior strictly depends on the interplay between the degrees of degeneracy of $D,F$ in $\pm1$.

\begin{prop}\label{prop:standing}
Let $D$ and $F$ be as in \eqref{eq:D-F}.
Then, there exists a unique solution (up to translation) to the problem \eqref{eq:ST}, and we have the three following alternatives:
\begin{enumerate}
\item if $n<m+2$, then $\omega_1,\omega_2\in\R$ and we say that the standing wave $\Phi$ touches the equilibria $\pm1$.
In particular, there exists a unique solution to \eqref{eq:ST} with $-\omega_1=\omega_2=\omega_\e\in(0,\infty)$, satisfying $\Phi(0)=0$ and $\Phi(\pm\omega_\e)=\pm1$, with
\begin{equation}\label{eq:omega}
	\omega_\e=\e\int_0^{+1}\sqrt{\frac{D(s)}{2F(s)}}\,ds,
\end{equation}
so that $\displaystyle\lim_{\e\to0^+}\omega_\e=0$.
\item If $n=m+2$, then $-\omega_1=\omega_2=\infty$, implying that the solutions to \eqref{eq:ST} never touch the equilibria $\pm1$, 
namely $|\Phi(x)|<1$ for any $x\in\R$ and we have the exponential decay
\begin{equation}\label{eq:exp-decay}
	\begin{aligned}
		&\Phi(x)+1\leq c_1e^{c_2 x}, \qquad\qquad&\mbox{ as }\, x\to-\infty,\\
		&1-\Phi(x)\leq c_2e^{-c_2 x}, \qquad\qquad&\mbox{ as }\, x\to+\infty,
	\end{aligned}
\end{equation}
for some constants $c_1,c_2>0$.
In particular, if we fix $\Phi(0)=0$ the solution is explicitly given by
\begin{equation}\label{eq:tanh}
	\Phi(x)=\tanh\left(\frac{x}{\sqrt{n}\e}\right)=\frac{1-\exp\left(-\frac{2}{\sqrt{n}\e}x\right)}{1+\exp\left(-\frac{2}{\sqrt{n}\e}x\right)}.
\end{equation}
\item Finally, if $n>m+2$ then, as in the previous case, the profile $\Phi$ never touches the equilibria $\pm1$ ($-\omega_1=\omega_2=\infty$),
but in this case we have only the algebraic decay:
\begin{equation}\label{eq:alg-decay}
	\begin{aligned}
		&\Phi(x)+1\leq c_1x^{-c_2}, &\mbox{as } x\to-\infty,\\
		&1-\Phi(x)\leq c_2 x^{-c_2}, \qquad &\mbox{as } x\to+\infty.
	\end{aligned}
\end{equation}
\end{enumerate}
\end{prop}

\begin{proof}
Assuming that $\Phi$ is sufficiently smooth and multiplying by $\Phi'$ the first equation of \eqref{eq:ST}, one obtains
\begin{equation*}
	\left[\frac{\e^2}2D(\Phi)(\Phi')^2-F(\Phi)\right]'=0, \qquad \qquad x\in(\omega_1,\omega_2),
\end{equation*}
which implies (together with the second and third conditions in \eqref{eq:ST}) that the left right hand side is constantly equal to zero.
Therefore, the existence of a solution to \eqref{eq:ST} can be deduced by studying the Cauchy problem
\begin{equation}\label{eq:Cauchy-ST}
	\begin{cases}
		\displaystyle\frac{\e^2}2D(\Phi)(\Phi')^2=F(\Phi), \qquad \qquad x\in(\omega_1,\omega_2),\\
		\Phi(0)=0,
	\end{cases}
\end{equation}
where we put the condition $\Phi(0)=0$ for definiteness (notice that we can assume $0\in(\omega_1,\omega_2)$ 
because the problem \eqref{eq:ST} is invariant under translations, 
that is if $\Phi(x)$ is a solution to \eqref{eq:ST}, then $\Phi(x+h)$ is still a solution for any $h\in\R$).
The solution to \eqref{eq:Cauchy-ST} is implicitly defined by
\begin{equation*}
	\int_0^{\Phi(x)}\sqrt{\frac{D(s)}{2F(s)}}\,ds=\frac{x}{\e}, \qquad x\in(\omega_1,\omega_2),
\end{equation*}
and so, the integrals
$$-\frac{\omega_1}{\e}=\int_{-1}^0\sqrt{\frac{D(s)}{2F(s)}}\,ds, \qquad \qquad \frac{\omega_2}{\e}=\int_0^{+1}\sqrt{\frac{D(s)}{2F(s)}}\,ds$$
play a crucial role in the behavior of the solution to \eqref{eq:Cauchy-ST}.
Precisely, the behavior of the functions $D,F$ close to $\pm1$ tells if $\omega_{1,2}$ are finite or infinite, 
and in the explicit case \eqref{eq:D-F}, we have
$$\int_{-1}^0\sqrt{\frac{D(s)}{2F(s)}}\,ds=\int_0^{+1}\sqrt{\frac{D(s)}{2F(s)}}\,ds=\sqrt{n}\int_0^1(1-s^2)^{\frac{m-n}2}\,ds<\infty \quad \Longleftrightarrow \quad n<m+2.$$
Moreover, if $n=m+2$, then \eqref{eq:Cauchy-ST} becomes 
\begin{equation*}
	\begin{cases}
		\e^2(\Phi')^2=\frac{1}{n}(1-\Phi^2)^2, \qquad \qquad  x\in(-\infty,+\infty),\\
		\Phi(0)=0,
	\end{cases}
\end{equation*}
and solving by separation of variable, one finds the explicit solution \eqref{eq:tanh}.
Finally, since
$$\e\Phi'=\frac{1}{\sqrt{n}}\left[(1-\Phi)(1+\Phi)\right]^\frac{n-m}2, \qquad \qquad x\in(-\infty,+\infty),$$
we have the algebraic decay \eqref{eq:alg-decay} if $n-m>2$. 
\end{proof}

In the proof of Proposition \ref{prop:standing}, we explicitly computed the solution of \eqref{eq:ST} only in the case $n=m+2$,
but it is easy to compute the explicit solution also in other cases. 
For instance, consider the case \eqref{eq:D-F} with $n=m$:
from \eqref{eq:Cauchy-ST} it follows that $(\e\Phi')^2=\frac1{n}$ and, as a simple consequence, we obtain the standing wave
\begin{equation}\label{eq:retta}
	\Phi(x)=\begin{cases}
		-1, \qquad \qquad & x\leq-\sqrt{n}\,\e, \\
		\displaystyle\frac{x}{\sqrt{n}\,\e}, & -\sqrt{n}\,\e< x<\sqrt{n}\,\e,\\
		1, &x\geq\sqrt{n}\,\e.
	\end{cases}
\end{equation}
Precisely, the profile $\Phi$ in \eqref{eq:retta} is a solution to \eqref{eq:ST} with $\Phi(0)=0$ and $-\omega_1=\omega_2=\sqrt{n}\,\e$.

As a second example, consider the case of $D$ and $F$ given by \eqref{eq:D-F} with $n=m+1$:
the solution to \eqref{eq:Cauchy-ST} satisfies $(\e\Phi')^2=\frac1{n}(1-\Phi^2)$, which is equivalent to $[\arcsin(\Phi)]'=\frac{1}{\sqrt{n}\,\e}$
and the profile
\begin{equation}\label{eq:sin}
	\Phi(x)=\begin{cases}
		-1, \qquad \qquad \qquad& x\leq-\frac\pi2\sqrt{n}\,\e,\\
		\displaystyle\sin\left(\frac{x}{\sqrt{n}\,\e}\right),  & -\frac\pi2\sqrt{n}\,\e<x<\frac\pi2\sqrt{n}\,\e, \\
		1, &x\geq\frac\pi2\sqrt{n}\,\e,
	\end{cases}
\end{equation}
is the solution to \eqref{eq:ST} with $-\omega_1=\omega_2=\frac\pi2\sqrt{n}\,\e$, satisfying $\Phi(0)=0$.

Proposition \ref{prop:standing} tells us that there is a crucial difference between the cases $n<m+2$ and $n\geq m+2$.
Indeed, if $n\geq m+2$ the standing wave $\Phi\in C^\infty(\R)$ and satisfies $\Phi'(x)>0$ for any $x\in\R$;
thus, a standing wave is never a stationary solution in any bounded interval with homogeneous Neumann boundary conditions,
that is we can never truncate $\Phi$ in a way to find a solution to \eqref{eq:stationary}, for some $a,b\in\R$.
On the contrary, in the case $n<m+2$ the standing wave $\Phi\in C^\infty(\R\backslash\left\{\omega_1,\omega_2\right\})$,
it satisfies $\Phi'(x)=0$ if $x<\omega_1$ or $x>\omega_2$ and we have the following alternatives:
if $m<n<m+2$, then $\Phi'=0$ at $\omega_1,\omega_2$ and, in particular, $\Phi\in C^2(\R)$ if $m+1<n<m+2$,
while $\Phi\in C^1(\R)$ if $m<n\leq m+1$
(see for example \eqref{eq:sin} where $\Phi\in C^1(\R)$ but $\Phi'$ is not differentiable at $\omega_1,\omega_2$);
if $n=m$, then $\Phi\in H^1(\R)$ is not differentiable at $\omega_1,\omega_2$, see \eqref{eq:retta};
finally, if $2\leq n<m$ then 
$$\lim_{x\to\omega_1^+}\Phi'(x)=\lim_{x\to\omega_2^-}\Phi'(x)=+\infty.$$
To be more precise, we have
\begin{equation*}
	\int_{\omega_1}^{\omega_2}\Phi'(x)^2\,dx=\frac{1}{\e\sqrt{n}}\int_{\omega_1}^{\omega_2}\left[1-\Phi^2(x)\right]^\frac{n-m}2\Phi'(x)\,dx=
	\frac{1}{\e\sqrt{n}}\int_{-1}^{+1}\left(1-s^2\right)^\frac{n-m}2\,ds,
\end{equation*}
and, as a consequence, $\Phi\in H^1(\R)$  if and only if $n>m-2$.
However, in any case we have $D(\Phi)(\Phi')^2=F(\Phi)=0$ at $\omega_1,\omega_2$
and if $n<m+2$, then a standing wave \eqref{eq:ST} is also a stationary solution to \eqref{eq:model} 
for any $[a,b]\supset[\omega_1,\omega_2]$, see Definition \ref{def:weak}.

Moreover, as we pointed out in Proposition \ref{prop:standing}, $\omega_\e=\mathcal{O}(\e)$ and so,
for any $[a,b]\subset\R$ we can take $\e$ sufficiently small that a standing wave \eqref{eq:ST} is a solution to \eqref{eq:stationary}.
As a consequence, the set of solutions to \eqref{eq:stationary} is very different whether $n<m+2$ or $n\geq m+2$ and in the first case we have many more solutions.
Indeed, all the solutions to \eqref{eq:stationary} satisfy
\begin{equation*}
	\frac{\e^2}{2}D(\varphi)(\varphi')^2-F(\varphi)=C,
\end{equation*}
for some $C\in\left[-\frac{1}{2n},0\right]$.
In particular, for $C=-\frac{1}{2n}$ we obtain the constant solution $\varphi\equiv0$, which is clearly a solution to \eqref{eq:stationary} for any $a,b$ and any $n,m$;
for $C\in(-\frac{1}{2n},0)$ it is easy to check that one has periodic solutions and so, there exist infinitely many periodic solutions to \eqref{eq:stationary} for any $a,b$ and any $n,m$;
finally, if $C=0$ it is important to distinguish the cases $n<m+2$ and $n\geq m+2$.
In both the cases, we have the constant solutions $\varphi\equiv\pm1$, but in the case $n\geq m+2$ we can not have other solutions to \eqref{eq:stationary}
(because the standing wave never satisfies the boundary conditions as we mentioned above), while if $n<m+2$ and $\e$ is small, 
then a standing wave is a solution to \eqref{eq:stationary} or, as we will see in details in Section \ref{sec:compactons},  
we can construct different stationary solutions to the equation \eqref{eq:model} in a bounded interval of the real line
satisfying homogeneous Neumann boundary conditions \eqref{eq:Neu} oscillating between (and touching) $-1$ and $+1$.
These solutions are known in literature as \emph{compactons} (cfr. \cite{CISc16} and references therein) and we will study its existence in the following section.

We conclude the study of standing wave for \eqref{eq:model} in the whole real line by considering the case \eqref{eq:D-F-onezero}.
Similarly to Proposition \ref{prop:standing}, we can prove the following result.
\begin{prop}\label{prop:standing-onezero}
Let $D$ and $F$ be as in \eqref{eq:D-F-onezero}.
Then, there exists a unique solution (up to translation) to the problem \eqref{eq:ST} with $\omega_1=-\infty$, and we have the three following alternatives:
\begin{enumerate}
\item if $n<m+2$, then $\omega_2<\infty$ and there exists a unique solution to \eqref{eq:ST} satisfying $\Phi(0)=0$, $\Phi(\omega_\e)=1$, with $\omega_\e$ given by \eqref{eq:omega}
and converging either exponentially fast to $-1$ as $x\to-\infty$ if $n=2$ (see the first estimate of \eqref{eq:exp-decay}) 
or algebraically fast if $n>2$ (see the first estimate of \eqref{eq:alg-decay}). 
\item If $n=m+2$, then $\omega_2=\infty$, i.e. the solutions to \eqref{eq:ST} never touch the equilibria $\pm1$, 
namely $|\Phi(x)|<1$ for any $x\in\R$ and it converges algebraically fast to $-1$ as $x\to-\infty$ (since $n>2$)
and exponentially fast to $+1$ as $x\to+\infty$.
\item Finally, if $n>m+2$, then we have the same situation of case (3) of Proposition \ref{prop:standing}.
\end{enumerate}
\end{prop}

\begin{proof}
The proof is very similar to the one of Proposition \ref{prop:standing}.
By proceeding in the same way, we infer
\begin{align*}
	&\omega_1=-\e\int_{-1}^0\sqrt{\frac{D(s)}{2F(s)}}\,ds=-\sqrt{n}\,\e\int_{-1}^0\frac{ds}{(1-s)^\frac{n-m}2(1+s)^\frac{n}2}=-\infty, \qquad \; \forall\,n\geq2,\\
	&\omega_2=\e\int_0^{+1}\sqrt{\frac{D(s)}{2F(s)}}\,ds=\sqrt{n}\,\e\int_{0}^1\frac{ds}{(1-s)^\frac{n-m}2(1+s)^\frac{n}2}<\infty \qquad \Longleftrightarrow \qquad n<m+2.
\end{align*}
Moreover, we have that a non-decreasing standing wave \eqref{eq:ST} satisfies 
\begin{equation}\label{eq:ode-Phi-onezero}
	\e\Phi'=\frac{1}{\sqrt{n}}(1-\Phi)^\frac{n-m}2(1+\Phi)^\frac n2, \qquad \qquad x\in(-\infty,\omega_2),
\end{equation}
and the proof is complete.
\end{proof}

As in the case \eqref{eq:D-F}, let us compute some explicit solutions in the case \eqref{eq:D-F-onezero}.
First, we consider the case $n=m=2$: equation \eqref{eq:ode-Phi-onezero} becomes $\e\Phi'=\frac1{\sqrt2}(1+\Phi)$, 
which is equivalent to $[\ln(1+\Phi)]'=\frac{1}{\sqrt{2}\,\e}$, and we deduce that the profile
\begin{equation}\label{eq:exp}
	\Phi(x)=\begin{cases}
		\displaystyle\exp\left(\frac{x}{\sqrt{2}\,\e}\right)-1, \qquad \qquad & x<(\sqrt{2}\ln2)\e, \\
		1, &x\geq(\sqrt{2}\ln2)\e.
	\end{cases}
\end{equation}
is the solution to \eqref{eq:ST} with $\omega_1=-\infty$, $\omega_2=(\sqrt{2}\ln2)\e$ and $\Phi(0)=0$.

Next, we consider the generic case $n=m>2$.
It follows that equation \eqref{eq:ode-Phi-onezero} becomes $\e\Phi'=\frac1{\sqrt{n}}(1+\Phi)^\frac n2$, 
which is equivalent to $[(1+\Phi)^{1-\frac n2}]'=\frac{2-n}{2\sqrt{n}\,\e}$, and we deduce that the profile
\begin{equation}\label{eq:alg-onezero}
	\Phi(x)=\begin{cases}
		\displaystyle\left(\frac{2\sqrt{n}\,\e}{(2-n)x+2\sqrt{n}\,\e}\right)^{\frac{2}{n-2}}-1, \qquad \qquad & x<\frac{2\sqrt{n}}{2-n} (2^{\frac{2-n}2}-1)\e, \\
		1, &x\geq\frac{2\sqrt{n}}{2-n} (2^{\frac{2-n}2}-1)\e.
	\end{cases}
\end{equation}
is the solution to \eqref{eq:ST} with $\omega_1=-\infty$, $\omega_2=\frac{2\sqrt{n}}{2-n} (2^{\frac{2-n}2}-1)\e$ and $\Phi(0)=0$.

Notice that in both the examples \eqref{eq:exp} and \eqref{eq:alg-onezero} the profile  $\Phi\in H^1(\R)$ is not differentiable in the point where it reaches $+1$;
however, in \eqref{eq:exp} $\Phi$ converges exponentially fast to $-1$ as $x\to-\infty$, while in \eqref{eq:alg-onezero} the speed of convergence is only algebraic.

In all this section, we have only considered non-decreasing standing wave, but we can define a \emph{non-increasing standing wave} in the similar way
and it is easy to check that by taking $-\Phi$ in \eqref{eq:tanh}-\eqref{eq:retta}-\eqref{eq:sin} and $\Phi(-x)$ in \eqref{eq:exp}-\eqref{eq:alg-onezero}, 
we obtain a \emph{non-increasing standing wave} connecting $+1$ and $-1$.

\subsection{Compactons}\label{sec:compactons}
In this section, we prove the existence of compactons for \eqref{eq:model}-\eqref{eq:Neu}, 
namely we construct some particular solutions to \eqref{eq:stationary}, when $D$ and $F$ are given by \eqref{eq:D-F} with $n<m+2$.
To be more precise, we show that for any $N\in\mathbb{N}$ and arbitrary points $h_1<h_2<\dots<h_N$, one can choose $\e>0$ so small that
there exist two stationary solutions oscillating between $\pm1$ (and touching them) and with exactly $N$ zeros located at $h_1<h_2<\dots<h_N$.

\begin{prop}\label{prop:ex-compactons}
Let $D$ and $F$ given by \eqref{eq:D-F} with $n<m+2$.
Fix $N\in\mathbb N$ and $N$ points satisfying $a<h_1<h_2<\dots<h_N<b$.
Then, there exists $\bar\e>0$ such that if $\e\in(0,\bar\e)$, then there exist two stationary solutions to \eqref{eq:model}-\eqref{eq:Neu}
with exactly $N$ zeros located at $h_1,h_2,\dots,h_N$ and oscillating between (touching) $-1$ and $+1$.
\end{prop}

\begin{proof}
As we have previously mentioned, this result follows from the behavior of the standing wave \eqref{eq:ST} described in Proposition \ref{prop:standing}.
In particular, consider the unique standing wave given by case (1) of Proposition \ref{prop:standing} and satisfying $\Phi(0)=0$.
We shall properly glue together some translations of the standing wave $\Phi$ and its reflection to obtain a solution to \eqref{eq:stationary}.

Define
$$m_1=a, \qquad m_i=\frac{h_{i-1}+h_i}{2}, \quad i=2,\dots,N, \qquad m_{N+1}=b,$$
and
\begin{equation}\label{eq:compactons-}
	\varphi_N^\e(x):=\Phi((-1)^{i+1}(x-h_i)), \qquad \qquad x\in[m_i,m_{i+1}], 
\end{equation}
for $i=1,\dots,N$.
Clearly, $\varphi_N^\e(h_i)=0$ for $i=1,\dots,N$ and as, we observed in Section \ref{sec:standing}
$\Phi(x-h_i)$ is a non-decreasing standing wave, while $-\Phi(x-h_i)=\Phi(h_i-x)$ is a non-increasing standing wave.
Since $\Phi(\pm\omega_\e)=\pm1$, where $\omega_\e$ is given by \eqref{eq:omega} and we want $\varphi_N^\e$ to touch both $-1$ and $+1$ 
in any interval $[m_i,m_{i+1}]$ with $\varphi'_N(a)=\varphi'_N(b)=0$, we require
$$\omega_\e<\min\left\{h_1-a,b-h_N\right\}, \qquad \mbox{ and } \qquad \omega_\e\leq\min_{2\leq i\leq N}\left\{\frac{h_i-h_{i-1}}{2}\right\}.$$
Denoting by
$$\delta_N:=\min\left\{h_1-a,b-h_N,\min_{2\leq i\leq N}\left\{\frac{h_i-h_{i-1}}{2}\right\}\right\},$$
we have $\varphi_N^\e(h_i\pm\omega_\e)=\pm(-1)^{i+1}$ if $\omega_\e\leq\delta_N$.
Since $\omega_\e$ is given by \eqref{eq:omega}, we define
$$\bar\e:=\delta_N\left(\int_0^{+1}\sqrt{\frac{D(s)}{2F(s)}}\,ds\right)^{-1}.$$
Let us stress that the constant $\bar\e\in(0,\infty)$ because $n>m+2$, and we have $\omega_\e<\delta_N$ for any $\e\in(0,\bar\e)$.
Therefore, the functions $\varphi_N^\e$ and $-\varphi_N^\e$ have been constructed by gluing together standing wave solutions \eqref{eq:ST},
they have $N$ zeros located at $h_1,h_2,\dots,h_N$, they satisfy the boundary conditions in \eqref{eq:stationary} 
and we can conclude that they are stationary solutions to \eqref{eq:model}-\eqref{eq:Neu}, see Definition \ref{def:weak}.
\end{proof}

In Figure \ref{fig:compactons} we show some examples of stationary solutions constructed in Proposition \ref{prop:ex-compactons}.
In particular, the profile in Figure \ref{fig:compactons} (left) can be obtained by using the formula \eqref{eq:compactons-} 
with standing wave $\Phi$ given by \eqref{eq:retta}.

\begin{figure}[htp]
\centering
\includegraphics[width=6.5cm,height=5.5cm]{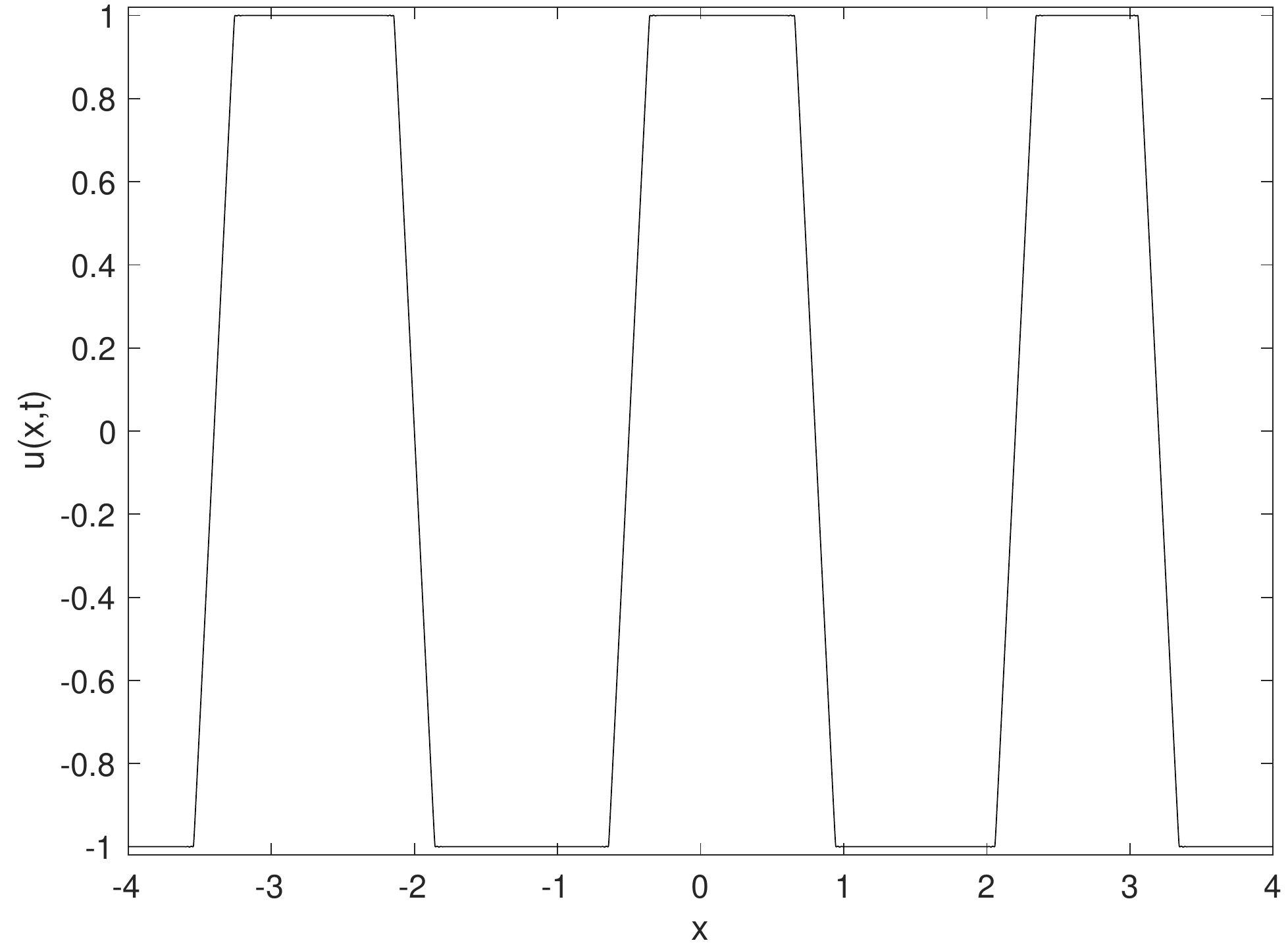}\qquad
\includegraphics[width=6.5cm,height=5.5cm]{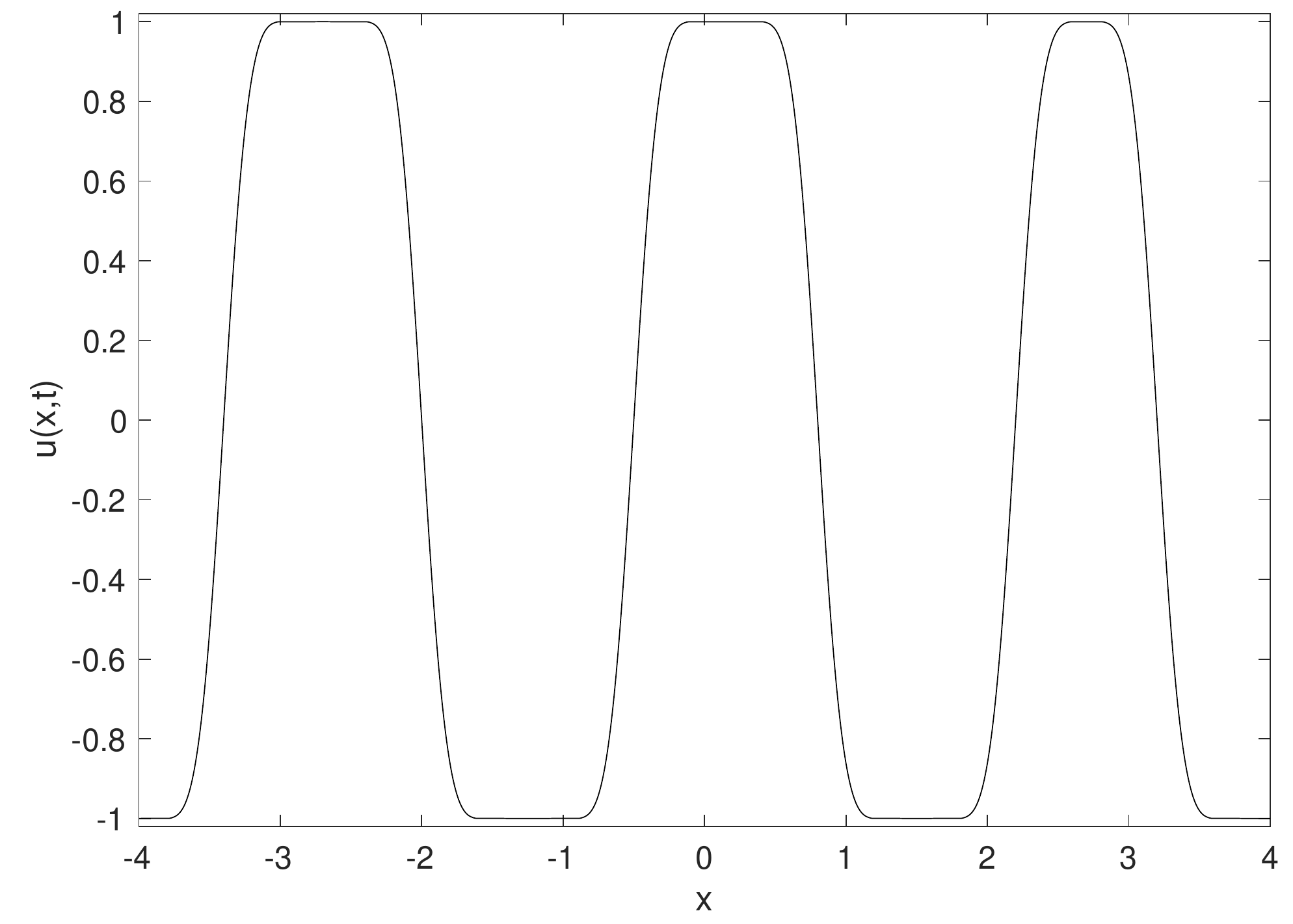}
\caption{Stationary solutions (compactons) to \eqref{eq:model}-\eqref{eq:Neu},
where $\e=0.1$ and the functions $D,F$ are given by \eqref{eq:D-F} with $m=n=2$ (left) and $m=1.5$, $n=3$ (right).
The zeros are located at $h_1=-3.4, h_2=-2, h_3=-0.5, h_4=0.8, h_5=2.2, h_6=3.2$.}
\label{fig:compactons}
\end{figure}

\begin{rem}
By construction, the function $\varphi_N^\e$ in \eqref{eq:compactons-} solves the first equation in \eqref{eq:stationary} almost everywhere
and it has the same regularity of the standing wave $\Phi$.
To be more precise, we can say that $\varphi_N^\e\in C^\infty\left(\R\backslash\left\{h_i\pm\omega_\e\right\}\right)$
and the regularity in the points $h_i\pm\omega_\e$ depends on $n,m$:
if $m<n<m+2$, the profile $\varphi_N^\e$ is (at least) differentiable in the whole interval $[a,b]$
($\varphi_N^\e\in C^2(a,b)$ for any $m+1<n<m+2$),
while in the case $n\leq m$, $\varphi_N^\e$ is not differentiable in the $2N$ points $h_i\pm\omega_\e$.
However, in any case $\varphi_N^\e$ is a stationary solution to \eqref{eq:model}-\eqref{eq:Neu}, see Definition \ref{def:weak}.
We also stress that if $\delta_N<\min\{h_1-a,b-h_N\}$, we can choose $\e=\bar\e$ 
or equivalently we can choose $\e$ so that $\omega_\e=\delta_N=\frac12(h_j-h_{j-1})$ for some $j\in\{2,\dots,N\}$, 
meaning that $h_{j-1}+\omega_\e=h_j-\omega_\e=m_j$ and we have two standing waves glued together exactly at the point where they reach $+1$ or $-1$. 
%On the other hand, if $\delta_N=\min\{h_1-a,b-h_N\}$ we can choose $\e=\bar\e$ only if $m<n<m+2$,
%otherwise $\varphi_N^\e$ does not satisfy the homogeneous Neumann boundary conditions in \eqref{eq:stationary}.
\end{rem}

It is important to mention that the function $\varphi^\e_N$ in \eqref{eq:compactons-} is a stationary solution to \eqref{eq:model}-\eqref{eq:Neu},
for any $N\in\mathbb N$ and for arbitrarily points $h_1,\dots,h_N$ provided that $\e$ is sufficiently small,
because $n<m+2$.
On the other hand, if we use the same construction with $\Phi$ given by Proposition \ref{prop:standing} in the case $n\geq m+2$ we obtain
a profile oscillating between $-1$ and $+1$, but it does not touch $\pm1$ and it is far from being a stationary solution
if the points $h_1,\dots,h_N$ are arbitrarily chosen and $\e>0$.
Indeed, the only solutions to \eqref{eq:stationary} in the case $n\geq m+2$ are constant or periodic. 
However, we will show in Section \ref{sec:slow} that a profile like \eqref{eq:compactons-} evolves very slowly in time
in the case $n\geq m+2$ and the solution starting with initial datum \eqref{eq:compactons-} maintains the same structure (at least)
for an exponentially long time, i.e. for a time  $T_\e\geq\nu\exp(c/\e)$ if $n=m+2$, and for an algebraically long time $T_\e\geq\nu\e^{-\beta}$, when $n>m+2$.

Concerning the case \eqref{eq:D-F-onezero}, the profile $\varphi_{N}(-x)$ in \eqref{eq:compactons-}, with standing wave given by Proposition \ref{prop:standing-onezero}
is a stationary solution if and only if $n<m+2$ and $N$ is an even number.

\begin{prop}\label{prop:stat-onezero}
Let $D$ and $F$ given by \eqref{eq:D-F-onezero} with $n<m+2$.
Fix $N\in\mathbb N$ and $2N$ points satisfying $a<h_1<h_2<\dots<h_{2N}<b$.
Then, there exists $\bar\e>0$ such that if $\e\in(0,\bar\e)$, then there exist a stationary solution to \eqref{eq:model}-\eqref{eq:Neu}
with exactly $2N$ zeros located at $h_1,h_2,\dots,h_{2N}$ and oscillating between $-1$ and $+1$ (touching only $+1$).
\end{prop}

\begin{proof}
The proof is the same of the one of Proposition \ref{prop:ex-compactons}.
Define 
$$\varphi_{2N}(x):=\Phi((-1)^{i}(x-h_i)), \qquad \qquad x\in[m_i,m_{i+1}],$$
for $i=1,\dots,2N$, with $\Phi$ given by Proposition \ref{prop:standing-onezero} in the case $n<m+2$.
If $\e$ is sufficiently small, then $\varphi_{2N}(a)=\varphi_{2N}(b)=+1$ and $\varphi'_{2N}(a)=\varphi'_{2N}(b)=0$,
because $2N$ is an even natural number.
\end{proof}
Therefore, in the case \eqref{eq:D-F-onezero} with $n<m+2$ we have only one stationary solution with an even number of zeros:
we can not have a stationary solution with $2N+1$ zeros because either $\varphi_{2N+1}(a)\in(-1,0)$ or $\varphi_{2N+1}(b)\in(-1,0)$,
so either $\varphi'_{2N+1}(a)\neq0$ or $\varphi'_{2N+1}(b)\neq0$ (the standing wave $\Phi$ never touches $-1$, see Proposition \ref{prop:standing-onezero}).

In Figure \ref{fig:one-degenerate}, we show two examples of profiles \eqref{eq:compactons-} with $\Phi$ given by \eqref{eq:exp}.
In both examples the number of transitions is even, but in the left one $\varphi_4(-4)=1$, while in the right one $\varphi_6(-4)<0$.
Hence, the profile in the left picture is a stationary solution to \eqref{eq:model}-\eqref{eq:Neu}, while the profile in the right one has
a \emph{non stationary transition layer structure}, which does not satisfy the boundary conditions \eqref{eq:Neu}.
We will show in Section \ref{sec:slow} that even if the latter profile is not stationary, it evolves very slowly in time
and the solution starting with such an initial datum maintains the same structure for a time
$T_\e\geq\nu\exp(c/\e)$ for some $c>0$, independent on $\e$.
\begin{figure}[htp]
\centering
\includegraphics[width=6.5cm,height=5.5cm]{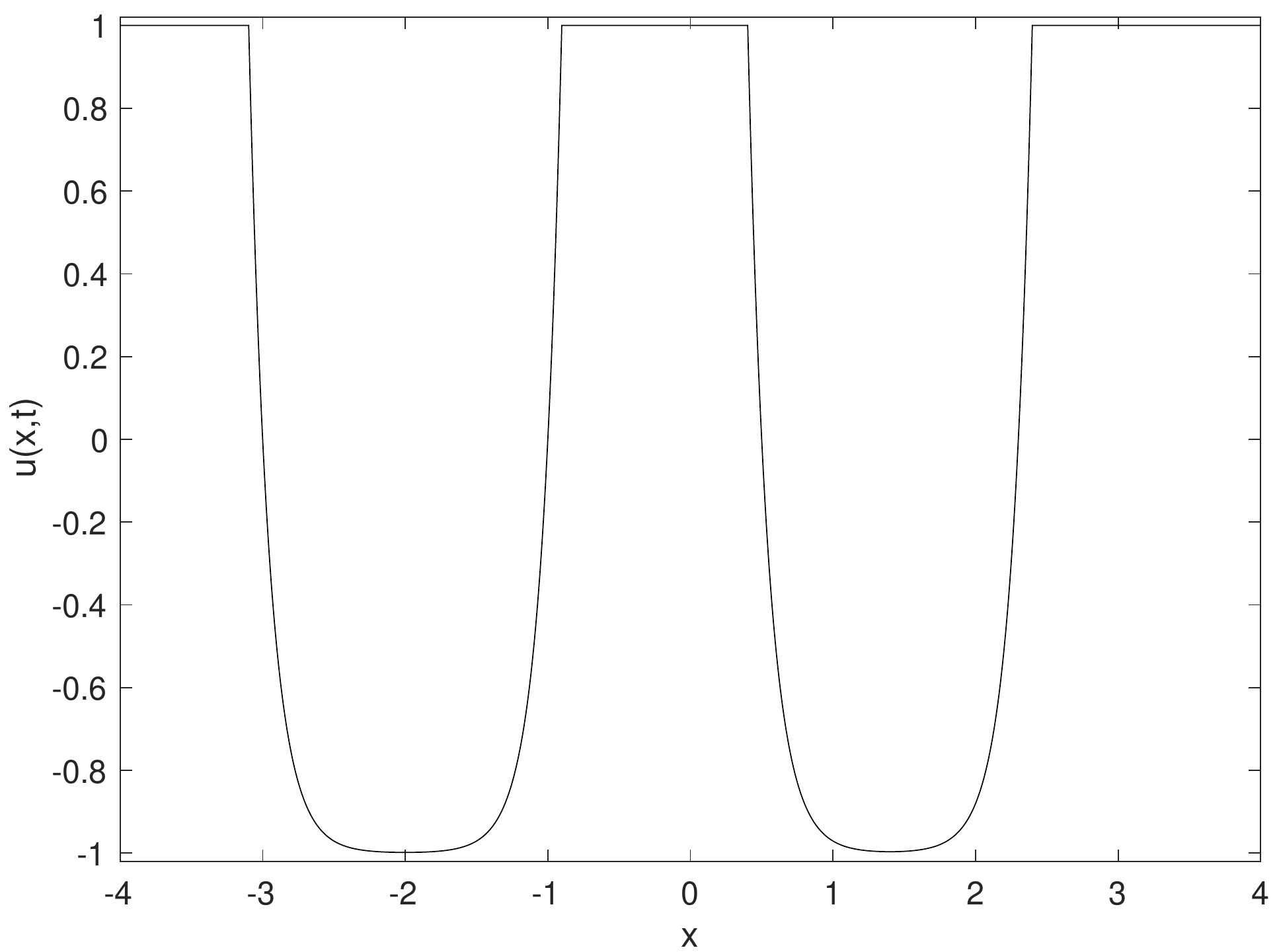}\qquad
\includegraphics[width=6.5cm,height=5.5cm]{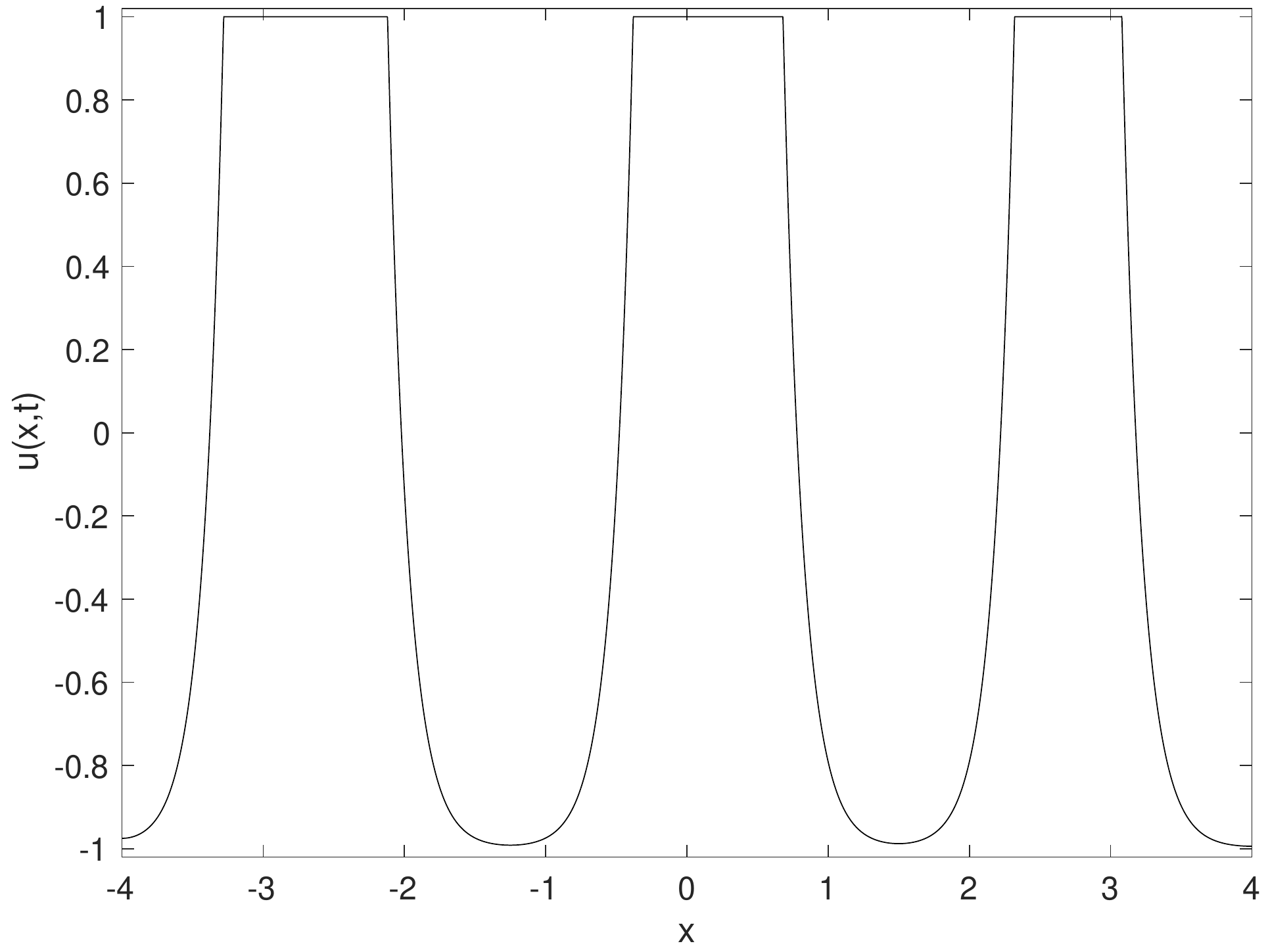}
\caption{Two different chains (see \eqref{eq:compactons-}) of standing waves $\Phi(x)$ and $\Phi(-x)$, 
where $\Phi$ is given by \eqref{eq:exp} with $\e=0.1$.
The zeros are located at $h_1=-3, h_2=-1, h_3=0.5, h_4=2.3$ in the left picture,
and at $h_1=-3.4, h_2=-2, h_3=-0.5, h_4=0.8, h_5=2.2, h_6=3.2$ in the right one.}
\label{fig:one-degenerate}
\end{figure}

\section{A variational result on the energy}
\label{sec:lower}

This section contains the key results to prove the slow motion of the solutions to \eqref{eq:model}-\eqref{eq:Neu}.
These results are purely variational in character and concern only the energy functional $E_\e$, defined in \eqref{eq:energy};
equation \eqref{eq:model} plays no role.
By applying Young inequality in \eqref{eq:energy}, we deduce that for any $a\leq c<d\leq b$,
\begin{equation}\label{eq:ineq}
	\int_c^d\left[\frac\e2D(u)u_x^2+\frac{F(u)}\e\right]\,dx \geq\left|\int_{u(c)}^{u(d)}\sqrt{2D(s)F(s)}\,ds\right|.
\end{equation}
In particular, if $u$ is an increasing function connecting the states $\pm1$, that is $u(c)=-1$, $u(d)=1$, we obtain
\begin{equation}\label{eq:gamma}
	\int_c^d\left[\frac\e2D(u)u_x^2+\frac{F(u)}\e\right]\,dx \geq\int_{-1}^{1}\sqrt{2D(s)F(s)}\,ds=:\gamma.
\end{equation}
The goal of this section is to prove that the constant $\gamma$ in \eqref{eq:gamma} is the minimum energy to have a transition between $-1$ and $+1$ in the following sense:
if a sequence $\{u^\e\}_{\e>0}\in H^1(a,b)$ converges in $L^1(a,b)$ as $\e\to0^+$ to a piecewise constant function 
\begin{equation}\label{eq:vstruct}
	v:[a,b]\rightarrow\{-1,1\}\  \hbox{with $N$ jumps located at } a<h_1<h_2<\cdots<h_N<b,
\end{equation}	
then 
\begin{equation}\label{eq:liminf-E}
	\liminf_{\e\to0^+} E_\e[u^\e]\geq N\gamma,
\end{equation}
and it is possibile to choose a sequence such that the equality holds. 
Actually, we shall prove a lower bound on the energy which implies \eqref{eq:liminf-E} and reads as
\begin{equation}\label{eq:lower-theta}
	E_\varepsilon[u]\geq N\gamma-\theta(\e),
\end{equation}
where $\theta(\e)\to0$ as $\e\to0^+$, provided that $L^1$-distance between $u$ and a function $v$ as in \eqref{eq:vstruct} is sufficiently small.
Then, we shall give an example for which in the limit \eqref{eq:liminf-E} the equality holds.

The main novelty of our result concerns the fact that the diffusion coefficient $D$ could be double degenerate \eqref{eq:D-F} or single degenerate \eqref{eq:D-F-onezero}.
In \cite{OwSt91}, the authors study the $\Gamma$-convergence properties of the general class of functionals 
\begin{equation*}
	W_\e[u]:=\int_{\Omega}\frac{1}{\e}w(x,u,\e\nabla u)\,dx,
\end{equation*}
where $\Omega$ is an open, bounded domain in $\R^n$ and the function $w\in C^3(\bar{\Omega}\times\R\times\R^n)$ 
satisfies appropriate assumptions (cfr. assumptions (H1)-(H7) in \cite{OwSt91}).
The functional \eqref{eq:energy} does not enter in the framework of \cite{OwSt91} if $D$ and $F$ are given by \eqref{eq:D-F} with $m>0$ and $n>2$,
because $D$ and $F$ are degenerate in the sense that $D(\pm1)=0$ and $F''(\pm1)=0$, respectively.
In \cite{FHLP20}, a lower bound of the form \eqref{eq:lower-theta} for a functional of the form \eqref{eq:energy} has been established,
but with the crucial assumptions $D(u)>0$ and $F''(\pm1)>0$.
Then, again if $D$ and $F$ are given by \eqref{eq:D-F} or \eqref{eq:D-F-onezero} with $m>0$ and $n>2$,
then the energy \eqref{eq:energy} does not satisfy the assumptions needed in \cite[Proposition 2.4]{FHLP20}.
Here, we extend the results of \cite{FHLP20}, by showing that the remainder $\theta$ in \eqref{eq:lower-theta} 
strictly depends on the degrees of the degeneracy of $D$ and $F$ in $\pm1$.
In particular, if $n=m+2$, then $\theta(\e)$ is exponentially small as $\e\to0^+$ (cfr. Proposition \ref{prop:lower} or Proposition \ref{prop:lower2}), 
while it is algebraically small when $n>m+2$ (cfr. Proposition \ref{prop:lower_deg} and Remark \ref{rem:alg}).

To start with, we prove the lower bound in the case $n=m+2$, generalizing the results of the case $m=0$, $n=2$, see \cite{FHLP20} or \cite{FLM19}.

\begin{prop}\label{prop:lower}
Consider the functional \eqref{eq:energy} with $D, F$ given by \eqref{eq:D-F}.
Fix $v$ as in \eqref{eq:vstruct} and $r>0$ such that
\begin{equation}\label{eq:r}
	h_i+r<h_{i+1}-r, \ \hbox{ for}\ i=1,\dots,N, \qquad   a\leq h_1-r,\qquad h_N+r\leq b.
\end{equation}
If $m>2$ and $n=m+2$ in \eqref{eq:D-F}, then for all $A\in(0,4r)$, there exist $\e_0,C,\delta>0$ such that if $u\in H^1(a,b)$ satisfies
\begin{equation}\label{eq:u-v}
	\|u-v\|_{{}_{L^1}}\leq\delta,
\end{equation}
then for any $\e\in(0,\e_0)$,
\begin{equation}\label{eq:lower}
	E_\varepsilon[u]\geq N\gamma-C\exp\left\{-A\sqrt{n}/\varepsilon\right\},
\end{equation}
where $\gamma$ is defined in \eqref{eq:gamma}.

If $m\in(1,2]$ and $n=m+2$ in \eqref{eq:D-F}, then for all $A\in(0,2r)$, there exist $\e_0,C,\delta>0$ such that if $u\in H^1(a,b)$ satisfies \eqref{eq:u-v},
then for any $\e\in(0,\e_0)$,
\begin{equation}\label{eq:lower2}
	E_\varepsilon[u]\geq N\gamma-C\exp\left\{-A/\varepsilon\right\}.
\end{equation}
Finally, in the case $1<m\leq n=2$, we have the same lower bound \eqref{eq:lower2} for all $A\in(0,2^\frac{2-m}2r)$
provided that $u\in H^1(a,b)$ satisfies \eqref{eq:u-v} with a sufficiently small $\delta$.
\end{prop}
\begin{proof}
First of all, we prove the case $n=m+2$ with $m>2$.
Fix $u\in H^1(a,b)$ satisfying \eqref{eq:u-v}, with $v,r$ as in \eqref{eq:vstruct} and \eqref{eq:r}.
Take $\hat r\in(0,r)$ and $\rho_1>0$ so small that 
\begin{equation}\label{eq:cond-A}
	A\leq\sqrt{2(2-\rho_1)\left(4-3\rho_1\right)}(r-\hat r).
\end{equation}
Then, choose $0<\rho_2<\rho_1$ sufficiently small that
\begin{equation}\label{eq:forrho2}
\begin{aligned}
	\int_{1-\rho_1}^{1-\rho_2}\sqrt{2D(s)F(s)}\,ds&>\int_{1-\rho_2}^{1}\sqrt{2D(s)F(s)}\,ds,  \\
	\int_{-1+\rho_2}^{-1+\rho_1}\sqrt{2D(s)F(s)}\,ds&> \int_{-1}^{-1+\rho_2}\sqrt{2D(s)F(s)}\,ds.
	\end{aligned}
\end{equation}
The choices of the constants $\hat r, \rho_1,\rho_2$ will be clear later on the proof.

We focus our attention on $h_i$, one of the discontinuous points of $v$ and, to fix ideas, 
let $v(h_i-r)=-1$, $v(h_i+r)=1$, the other case being analogous.
We can choose $\delta>0$ so small in \eqref{eq:u-v} so that there exist $r_+$ and $r_-$ in $(0,\hat r)$ such that
\begin{equation}\label{2points}
	|u(h_i+r_+)-1|<\rho_2, \qquad \quad \mbox{ and } \qquad \quad |u(h_i-r_-)+1|<\rho_2.
\end{equation}
Indeed, assume by contradiction that $|u-1|\geq\rho_2$ throughout $(h_i,h_i+\hat r)$; then
\begin{equation*}
	\delta\geq\|u-v\|_{{}_{L^1}}\geq\int_{h_i}^{h_i+\hat r}|u-v|\,dx\geq\hat r\rho_2,
\end{equation*}
and this leads to a contradiction if we choose $\delta\in(0,\hat r\rho_2)$.
Similarly, one can prove the existence of $r_-\in(0,\hat r)$ such that $|u(h_i-r_-)+1|<\rho_2$.

Now, we consider the interval $(h_i-r,h_i+r)$ and claim that
\begin{equation}\label{eq:claim}
	\int_{h_i-r}^{h_i+r}\left[\frac\e2D(u)u_x^2+\frac{F(u)}\e\right]\,dx\geq \gamma-\tfrac{C}N\exp\left\{-A\sqrt{n}/\varepsilon\right\},
\end{equation}
for some $C>0$ independent on $\e$.
If $u(h_i+r_+)\geq1$ and $u(h_i-r_-)\leq-1$, then from \eqref{eq:ineq} we can conclude that
\begin{equation*}
	\int_{h_i-r_-}^{h_i+r_+}\left\{\frac\e2[D(u)u_x]^2+\frac{F(u)}\e\right\}\,dx\geq \gamma,
\end{equation*}
which implies \eqref{eq:claim}.
On the other hand, notice that in general we have
\begin{align}
	\int_{h_i-r}^{h_i+r}\left[\frac\e2D(u)u_x^2+\frac{F(u)}\e\right]\,dx & 
	\geq \int_{h_i+r_+}^{h_i+r}\left[\frac\e2D(u)u_x^2+\frac{F(u)}\e\right]\,dx\notag \\ 
	& \quad + \int_{h_i-r}^{h_i-r_-}\left[\frac\e2D(u)u_x^2+\frac{F(u)}\e\right]\,dx \notag \\
	& \quad +\int_{-1}^{1}\sqrt{2D(s)F(s)}\,ds\notag\\
	&\quad-\int_{-1}^{u(h_i-r_-)}\sqrt{2D(s)F(s)}\,ds \notag \\
	& \quad-\int_{u(h_i+r_+)}^{1}\sqrt{2D(s)F(s)}\,ds\notag \\
	&=:I_1+I_2+\gamma-I_3-I_4, \label{eq:Pe}
\end{align}
where we again used \eqref{eq:ineq}. 
Regarding $I_1$, recall that $1-\rho_2<u(h_i+r_+)<1$ and consider the unique minimizer $z:[h_i+r_+,h_i+r]\rightarrow\R$ 
of $I_1$ subject to the boundary condition $z(h_i+r_+)=u(h_i+r_+)$.
If the range of $z$ is not contained in the interval $(1-\rho_1,1+\rho_1)$, then from \eqref{eq:ineq}, it follows that
\begin{equation}\label{E>fi}
	\int_{h_i+r_+}^{h_i+r}\left[\frac\e2D(z)z_x^2+\frac{F(z)}\e\right]\,dx>\int_{u(h_i+r_+)}^{1}\sqrt{2D(s)F(s)}\,ds=I_4,
\end{equation}
where, in the last inequality, we used the first estimate of \eqref{eq:forrho2} and, so the smallness of $r_+$ and $\rho_2$.
Suppose, on the other hand, that the range of $z$ is contained in the interval $(1-\rho_1,1+\rho_1)$. 
Then, the Euler-Lagrange equation for $z$ is
\begin{align*}
	&\e D(z)z''=\e^{-1}F'(z)-\frac\e2D'(z)(z')^2, \quad \qquad x\in(h_i+r_+,h_i+r),\\
	&z(h_i+r_+)=u(h_i+r_+), \quad \qquad z'(h_i+r)=0.
\end{align*}
For later use, we multiply by $z'$ the latter equation to get
\begin{equation*}
	\left[\frac{\e^2}2D(z)(z')^2-F(z)\right]'=0,
\end{equation*}
which implies
\begin{equation*}
	\frac{\e^2}2D(z(x))z'(x)^2=F(z(x))-F(z(h_i+r)),
\end{equation*}
and so,
\begin{equation}\label{eq:stima-z'}
	\frac{\e^2}2D(z(x))z'(x)^2\leq F(z(x)).
\end{equation}
Denoting by $\psi(x):=(z(x)-1)^2$, we have $\psi'=2(z-1)z'$ and 
\begin{equation*}
	\psi''=2(z-1)z''+2(z')^2=\frac{2F'(z)}{\varepsilon^2 D(z)}(z-1)+\left[\frac{2D(z)-(z-1)D'(z)}{D(z)}\right](z')^2.
\end{equation*}
By using the explicit formulas for $D$ and $F$, erasing the modulus because the range of the minimizer $z$ is contained in $(1-\rho_1,1)$, we obtain
\begin{align}
	\psi''&=-\frac{2z(1-z^2)^{n-1}}{\e^2(1-z^2)^{m}}(z-1)+\left[\frac{2(1-z)^{m}(1+z)^{m}-2mz(1+z)^{m-1}(1-z)^{m}}{(1-z^2)^{m}}\right](z')^2 \notag\\
	&=\frac{2z(1+z)^{n-m-1}(1-z)^{n-m}}{\e^2}+2\left[\frac{1+z(1-m)}{1+z}\right](z')^2.\label{eq:psi''}
\end{align}
Using again that the range of $z$ is contained in the interval $(1-\rho_1,1)$, we can choose $\rho_1$ sufficiently small so that
the coefficient of $(z')^2$ is non-positive (because $m>2$) and, by using \eqref{eq:stima-z'}, we deduce
\begin{align}
	\psi''&\geq\frac{2z(1+z)^{n-m-1}(1-z)^{n-m}}{\e^2}+\frac{4}{\e^2}\left[\frac{1+z(1-m)}{1+z}\right]\frac{F(z)}{D(z)}\notag \\
	&=\frac{2z(1+z)^{n-m-1}(1-z)^{n-m}}{\e^2}+\frac{2}{n\e^2}\left[\frac{1+z(1-m)}{1+z}\right](1-z^2)^{n-m}\notag \\
	&=\frac{2(1+z)^{n-m-1}(1-z)^{n-m}}{n\e^2}\left[1+z\left(n-m+1\right)\right]\notag \\
	&=\frac{2(1+z)^{n-m-1}}{n\e^2}\left[1+z\left(n-m+1\right)\right]\psi^\frac{n-m}2\notag\\
	&\geq\frac{2(2-\rho_1)\left(4-3\rho_1\right)}{n\e^2}\psi, \label{eq:n-m}
\end{align}
where we used $n=m+2$.
By using \eqref{eq:cond-A} and denoting by $\mu=A/(r-\hat r)$, we get from \eqref{eq:n-m}
\begin{align*}
	\psi''(x)-\frac{\mu^2}{n\varepsilon^2}\psi(x)\geq0, \quad \qquad x\in(h_i+r_+,h_i+r),\\
	\psi(h_i+r_+)=(u(h_i+r_+)-1)^2, \quad \qquad \psi'(h_i+r)=0.
\end{align*}
We compare $\psi$ with the solution $\hat \psi$ of
\begin{align*}
	\hat\psi''(x)-\frac{\mu^2}{n\varepsilon^2}\hat\psi(x)=0, \quad \qquad x\in(h_i+r_+,h_i+r),\\
	\hat\psi(h_i+r_+)=(u(h_i+r_+)-1)^2, \quad \qquad \hat\psi'(h_i+r)=0,
\end{align*}
which can be explicitly calculated to be
\begin{equation*}
	\hat\psi(x)=\frac{(u(h_i+r_+)-1)^2}{\cosh\left[\frac{\mu}{\sqrt{n}\varepsilon}(r-r_+)\right]}\cosh\left[\frac{\mu}{\sqrt{n}\varepsilon}(x-h_i-r)\right].
\end{equation*}
By the maximum principle, $\psi(x)\leq\hat\psi(x)$ so, in particular,
\begin{equation*}
	\psi(h_i+r)\leq\frac{(u(h_i+r_+)-1)^2}{\cosh\left[\frac{\mu}{\sqrt{n}\varepsilon}(r-r_+)\right]}\leq2\exp(-A/\sqrt{n}\varepsilon)(u(h_i+r_+)-1)^2,
\end{equation*}
where we used $\mu=A/(r-\hat r)$ and $r_+\in(0,\hat r)$.
Then, we have 
\begin{equation}\label{|z-v+|<exp}
	|z(h_i+r)-1|\leq\sqrt2\exp(-A/2\sqrt{n}\varepsilon)\rho_2.
\end{equation}
Using \eqref{|z-v+|<exp}, we deduce
\begin{equation} \label{fi<exp}
	\left|\int_{z(h_i+r)}^{1}\sqrt{2F(s)D(s)}\,ds\right|\leq C_n\exp\left\{-A\sqrt{n}/\varepsilon\right\},
\end{equation}
where the positive constant $C_n$ depends on $n$.
From \eqref{eq:ineq}-\eqref{fi<exp} it follows that, for some constant $C>0$, 
\begin{align}
	\int_{h_i+r_+}^{h_i+r}\left[\frac\e2D(z)z_x^2+\frac{F(z)}\e\right]\,dx &\geq \left|\int_{z(h_i+r_+)}^{1}\sqrt{2F(s)D(s)}\,ds\,-\right.\nonumber \\
	&\qquad \qquad\left.\int_{z(h_i+r)}^{1}\sqrt{2F(s)D(s)}\,ds\right| \nonumber\\
	& \geq I_4-\tfrac{C}{2N}\exp\left\{-A\sqrt{n}/\varepsilon\right\}. \label{E>fi-exp}
\end{align}
Combining \eqref{E>fi} and \eqref{E>fi-exp}, we get that the constrained minimizer $z$ of the proposed variational problem satisfies
\begin{equation*}	
	\int_{h_i+r_+}^{h_i+r}\left[\frac\e2D(z)z_x^2+\frac{F(z)}\e\right]\,dx \geq I_4-\tfrac{C}{2N}\exp\left\{-A\sqrt{n}/\varepsilon\right\}.
\end{equation*}
The restriction of $u$ to $[h_i+r_+,h_i+r]$ is an admissible function, so it must satisfy the same estimate and we have
\begin{equation}\label{eq:I1}
	I_1\geq I_4-\tfrac{C}{2N}\exp\left\{-A\sqrt{n}/\varepsilon\right\}.
\end{equation}
The term $I_2$ on the right hand side of \eqref{eq:Pe} is estimated similarly by analyzing the interval $[h_i-r,h_i-r_-]$ 
and using the second condition of \eqref{eq:forrho2} to obtain the corresponding inequality \eqref{E>fi}.
The obtained lower bound reads:
\begin{equation}\label{eq:I2}	
	I_2\geq I_3-\tfrac{C}{2N}\exp\left\{-A\sqrt{n}/\varepsilon\right\}.
\end{equation}
Finally, by substituting \eqref{eq:I1} and \eqref{eq:I2} in \eqref{eq:Pe}, we deduce \eqref{eq:claim}. 
Summing up all of these estimates for $i=1, \dots, N$, namely for all transition points, we end up with
\begin{equation*}
	E_\varepsilon[u]\geq\sum_{i=1}^N\int_{h_i-r}^{h_i+r}\left[\frac\e2D(u)u_x^2+\frac{F(u)}\e\right]\,dx\geq N\gamma-C\exp\left\{-A\sqrt{n}/\varepsilon\right\},
\end{equation*}
and the proof of \eqref{eq:lower} is complete.

The proof of the cases $n=m+2$ with $m\in(1,2]$ and $1<m\leq n=2$ are very similar.
In the case $n=m+2$ with $m\in(1,2]$, the only difference is that instead of \eqref{eq:cond-A}, 
one has to choose $\hat r\in(0,r)$ and $\rho_1>0$ so small that $A\leq\sqrt{2(1-\rho_1)(2-\rho_1)}(r-\hat r)$
and then, since the coefficient behind $(z')^2$ in \eqref{eq:psi''} is non-negative, (being the range of $z$ contained in $(1-\rho_1,1)$ and $m\leq2$),
we can conclude
\begin{equation*}
	\psi''\geq\frac{2z(1+z)(1-z)^{2}}{\e^2}\geq2(1-\rho_1)(2-\rho_1)\psi\geq \frac{\mu^2}{\e^2}\psi,
\end{equation*}
where $\mu=A/(r-\hat r)$.

Similarly, in the case $1<m\leq n=2$, one can choose $\hat r\in(0,r)$ and $\rho_1>0$ so small that $A\leq\sqrt{2(1-\rho_1)(2-\rho_1)^{1-m}}(r-\hat r)$
and conclude that
\begin{equation*}
	\psi''\geq\frac{2z(1+z)^{1-m}(1-z)^{2-m}}{\e^2}\geq2(1-\rho_1)(2-\rho_1)^{1-m}\psi^{1-\frac m2}\geq \frac{\mu^2}{\e^2}\psi,
\end{equation*}
where $\mu=A/(r-\hat r)$.
Once the estimate $\psi''\geq\frac{\mu^2}{\e^2}\psi$ is established, the proof of the lower bound \eqref{eq:lower2} is obtained as the one of \eqref{eq:lower}.
\end{proof}

\begin{rem}
In Proposition \ref{prop:lower}, we prove the lower bound for the energy \eqref{eq:lower-theta} 
with an \emph{exponentially small term} $\theta$ as $\e\to0^+$ in the cases $n=m+2$ and $1<m\leq n=2$.
A similar lower bound can be easily extended to the case $0<m-2\leq n\leq m+2$:
for such a purpose, in \eqref{eq:n-m} we used only at the last passage that $n=m+2$ so that the interested reader may easily verify that 
an estimate of the type $\psi''\geq\frac{\mu^2}{\e^2}\psi$ can be obtained in general for $0<m-2\leq n\leq m+2$.
\end{rem}

\begin{rem}
The constant $\gamma$ in \eqref{eq:gamma} can be explicitly computed when $D$ and $F$ are given by \eqref{eq:D-F} and it is given by
\begin{equation*}
	\gamma=\gamma_{n,m}=\frac{1}{\sqrt{n}}\int_{-1}^{1}(1-s^2)^\frac{n+m}2\,ds=\frac{\sqrt\pi\,\Gamma(\frac{n+m}2)}{\sqrt n\,\Gamma(\frac{n+m+3}2)},
\end{equation*}
where $\Gamma$ is the Euler's gamma function.
In particular, when $\frac{m+n}2\in\mathbb{N}$ we have
$$\gamma_{n,m}=\frac{2[2^\frac{n+m}2(\frac{n+m}2)!]^2}{\sqrt{n}(n+m+1)!},$$
and if $n=m+2$ with $m\in\mathbb N$, then
$$\gamma_{n,n-2}=\frac{[2^{n}(n-1)!]^2}{2\sqrt{n}(2n-1)!}.$$
\end{rem}

For simplicity, in Proposition \ref{prop:lower} we consider the case of even functions $D$ and $F$,
but the result can be extended to the general case
$$D(u)=|1-u|^{m_1}|1+u|^{m_2}, \qquad \qquad F(u)=|1-u|^{n_1}|1+u|^{n_2}.$$
In particular, in the case $m_1\in(1,2]$, $m_2=0$ and $n_1=n_2=2$ (that is \eqref{eq:D-F-onezero} with $m\in(1,2]$ and $n=2$), we obtain the following result.

\begin{prop}\label{prop:lower2}
Consider the functional 
\begin{equation*}
	E_\e[u]:=\int_a^b\left[\frac\e2|1-u|^{m}u_x^2+\frac{|1-u^2|^2}{4\e}\right]\,dx,
\end{equation*}
with $m\in(1,2]$, and fix a piecewise constant function $v$ and $r>0$ as in \eqref{eq:vstruct}-\eqref{eq:r}.
Then for all $A\in(0,2^\frac{2-m}2r)$, there exist $\e_0,C,\delta>0$ such that if $u\in H^1(a,b)$ satisfies \eqref{eq:u-v},
then for any $\e\in(0,\e_0)$, we have
\begin{equation}\label{eq:1-deg-cm}
	E_\varepsilon[u]\geq N\gamma-C\exp\left\{-A/\varepsilon\right\},
\end{equation}
where $\displaystyle \gamma:=\frac{1}{\sqrt2}\int_{-1}^1(1-s)^\frac{m+2}2(1+s)\,ds=\frac{4\cdot2^{\frac{m+5}2}}{(m+4)(m+6)}$.
\end{prop}

Next, we prove two lower bounds for the energy \eqref{eq:lower-theta} with an \emph{algebraically reminder} $\theta(\e)$:
the first one concerns the case \eqref{eq:D-F} and it is very important in the case $n>m+2$, 
the second one regards the case \eqref{eq:D-F-onezero}.

\begin{prop}\label{prop:lower_deg}
Consider the functional \eqref{eq:energy} with $D, F$ given by \eqref{eq:D-F}.
Let $v:(a,b)\rightarrow\{-1,+1\}$ a piecewise constant function with exactly $N$ discontinuities (as in \eqref{eq:vstruct}) and define the sequence
\begin{equation}\label{eq:exp_alg}
	\begin{cases}
		k_1=0,\\
		k_2:=\alpha, \\
         	k_{j+1}:=\alpha(k_{j}+1), \qquad j\geq2, 
	\end{cases}
	\quad \mbox{ where }  \quad \alpha:=\displaystyle\frac{n+m+2}{2n}.
\end{equation}
Then, for any $j\in \mathbb N$ there exist constants $\delta_j>0$ and $C_j>0$ such that if $w\in H^1$ satisfies
\begin{equation}\label{|w-v|_L^1<delta_l}
	\|w-v\|_{L^1}\leq\delta_j,
\end{equation}
and
\begin{equation}\label{E_p(w)<Nc0+eps^l}
	E_\e[w]\leq N\gamma+\varepsilon^{k_{j}},
\end{equation}
with $\varepsilon$ sufficiently small, then
\begin{equation}\label{E_p(w)>Nc0-eps^l}
	E_\e[w]\geq N\gamma-C_j\varepsilon^{k_{j+1}},
\end{equation}
where $\gamma$ is defined in \eqref{eq:gamma}.
\end{prop}

\begin{proof}
We prove our statement by induction on $j\geq1$ and we begin our proof by considering the case of only one transition ($N=1$). 
Let $h_1$ be the only point of discontinuity of $v$ and assume, without loss of generality, that $v=-1$ on $(a,h_1)$. 
Also, we choose $\delta_j$ small enough such that
\begin{equation*}
	(h_1-2j\delta_j,h_1+2j\delta_j)\subset(a,b).
\end{equation*}
Our goal is to show that for any $j\in\mathbb N$ there exist $x_j\in(h_1- 2j\delta_j,h_1)$ and $y_j\in(h_1, h_{1}+2j\delta_j)$ such that
\begin{equation}\label{x_k,y_k}
	w(x_j)\leq-1+C_j\varepsilon^\frac{k_j+1}{n}, \qquad w(y_j)\geq 1-C_j\varepsilon^\frac{k_j+1}{n},
\end{equation}
and
\begin{equation}\label{E_p(w)-xk,yk}
	\int_{x_j}^{y_j}\left[\frac\e2D(u)u_x^2+\frac{F(u)}{\e}\right]\,dx\geq\gamma-C_j\varepsilon^{k_{j+1}},
\end{equation}
where $\{k_j\}_{{}_{j\geq1}}$ is defined in \eqref{eq:exp_alg}.
We start with the base case $j=1$, and we show that hypotheses \eqref{|w-v|_L^1<delta_l} and \eqref{E_p(w)<Nc0+eps^l} imply the existence of two points 
$x_1{\in(h_1-2\delta_j,h_1)}$ and $y_1\in(h_1,h_1+2\delta_j)$ such that
\begin{equation}\label{x_1,y_1}
	w(x_1)\leq-1+C_j\varepsilon^\frac1{n}, \qquad w(y_1)\geq1-C_j\varepsilon^\frac1{n}.
\end{equation}
Here and throughout, $C_j$ represents a positive constant that is independent of $\varepsilon$,
whose value may change from line to line. 
From hypothesis \eqref{|w-v|_L^1<delta_l} we have
\begin{equation}\label{int su (gamma,b)}
	\int_{h_1}^b|w-1|\leq\delta_j,
\end{equation}
so that, denoting by $S^-:=\{y:w(y)\leq0\}$ and by $S^+:=\{y:w(y)>0\}$, 
 \eqref{int su (gamma,b)} yields
\begin{equation*}
\begin{aligned}
\textrm{meas}(S^-\cap(h_1,b))\leq\delta_j \qquad \mbox{and} \qquad
	\textrm{meas}(S^+\cap(h_1,h_1+2\delta_j))\geq\delta_j.
	\end{aligned}
\end{equation*}
Furthermore, from \eqref{E_p(w)<Nc0+eps^l}  with $j=1$, we obtain
\begin{equation*}
	\int_{S^+\cap(h_1,h_1+2\delta_j)}\frac{F(w)}\varepsilon\, dx\leq\gamma+1,
\end{equation*}
and therefore there exists $y_1\in S^+\cap(h_1,h_1+2\delta_j)$ such that
\begin{equation*}
	F(w(y_1))\leq C_j\varepsilon, \qquad \quad C_j=\frac{\gamma+1}{\delta_j}.
\end{equation*}
From the definition of $F$ in \eqref{eq:D-F}, it follows that $w(y_1)\geq 1-C_j\e^{\frac1{n}}$.
The existence of $x_1{\in S^-\cap(h_1-2\delta_j,h_1)}$ such that $w(x_1)\leq-1+C_j\varepsilon^\frac1{n}$ can be proved similarly.

Now, let us prove that \eqref{x_1,y_1} implies \eqref{E_p(w)-xk,yk} in the case $j=1$,
and as a trivial consequence we obtain the statement \eqref{E_p(w)>Nc0-eps^l} with $j=1$ and $N=1$.
Indeed, by using \eqref{eq:ineq} and \eqref{x_1,y_1} one deduces 
\begin{equation}\label{stima-l=1}
	\begin{aligned}
		E_\e[w]&\geq\int_{x_1}^{y_1}\left[\frac\e2D(w)w_x^2+\frac{F(w)}{\e}\right]\,dx
		\geq\int_{w(x_1)}^{w(y_1)} \sqrt{2D(s)F(s)}\,ds\\
		&\geq\gamma-\int_{1-C_j\varepsilon^\frac1{n}}^1 \sqrt{2D(s)F(s)}\,ds-\int_{-1}^{-1+C_j\varepsilon^\frac1{n}}\sqrt{2D(s)F(s)}\,ds\\
		&\geq\gamma-C_j\e^\alpha,
	\end{aligned}
\end{equation}
where $\alpha$ is defined in \eqref{eq:exp_alg} and we used the formulas for $D,F$ \eqref{eq:D-F} (actually, it is important only the behavior close to $\pm1$).
This concludes the proof in the case $j=1$ with one transition ($N=1$).

We now enter the core of the induction argument, proving that if \eqref{E_p(w)-xk,yk} holds true for for any $i\in\{1,\dots,j-1\}$, $j\geq2$, then \eqref{x_k,y_k} holds true.
By using \eqref{|w-v|_L^1<delta_l} we have
\begin{equation}\label{meas>delta_l-k}
	\textrm{meas}(S^+\cap(y_{j-1},y_{j-1}+2\delta_j))\geq\delta_j.
\end{equation}
Furthermore, by using  \eqref{E_p(w)<Nc0+eps^l} and \eqref{E_p(w)-xk,yk} in the case $j-1$, we deduce
\begin{equation*}
	\int_{y_{j-1}}^b\frac{F(w)}\varepsilon dx\leq C_j\varepsilon^{k_j},
\end{equation*}
implying
\begin{equation}\label{int F(w)<Ceps^k+1}
	\int_{S^+\cap(y_{j-1},y_{j-1}+2\delta_j)} F(w)\,dx\leq C_j\varepsilon^{k_j+1}.
\end{equation}
Finally, from  \eqref{meas>delta_l-k} and \eqref{int F(w)<Ceps^k+1} there exists $y_{j}\in S^+\cap(y_{j-1},y_{j-1}+2\delta_j)$ such that
\begin{equation*}
	F(w(y_{j}))\leq \frac{C_j}{\delta_j}\varepsilon^{k_j+1},
\end{equation*}
and, as a consequence, we have the existence of $y_{j}\in(y_{j-1},y_{j-1}+2\delta_j)$ as in \eqref{x_k,y_k}. 
The existence of $x_{j}\in(x_{j-1}-2\delta_j,x_{j-1})$ can be proved similarly. 

Reasoning as in \eqref{stima-l=1}, one can easily check that \eqref{x_k,y_k} implies
\begin{equation*}
	\int_{x_{j}}^{y_{j}}\left[\frac\e2D(w)w_x^2+\frac{F(w)}{\e}\right]\,dx\geq\gamma-C_j\varepsilon^{k_{j+1}},
\end{equation*}
and  the induction argument is completed, as well as the proof in case $N=1$.

The previous argument can be easily adapted to the case $N>1$. 
Let $v$ be as in \eqref{eq:vstruct}, and set $a=h_0, h_{N+1}=b$. 
We argue as in the case $N=1$ in each point of discontinuity $h_i$, by choosing
the constant $\delta_j$ so that 
$$
	{h_i}+2j\delta_j<h_{i+1}-2j\delta_j,\qquad \quad 0\leq i\leq N,
$$
and by assuming, without loss of generality, that $v=-1$ on $(a,h_1)$. 
Proceeding as in \eqref{x_1,y_1}, one can obtain the existence of $x^i_1\in(h_i-2\delta_j,h_i)$ and $y^i_1\in(h_i,h_i+2\delta_j)$ such that
\begin{align*}
	w(x^i_1)&\approx (-1)^i, &w(y^i_1)&\approx(-1)^{i+1},\\
	F(w(x^i_1))&\leq C\varepsilon, &F(w(y^i_1))&\leq C\varepsilon.
\end{align*}
On each interval $(x_1^i,y_1^i)$ we estimate as in \eqref{stima-l=1}, so that by summing  one obtains
$$
	\sum_{i=1}^N\int_{x_1^i}^{y_1^i}\left[\frac\e2D(w)w_x^2+\frac{F(w)}{\e}\right]\,dx\geq N\gamma-C\varepsilon^\alpha,
$$
that is \eqref{E_p(w)>Nc0-eps^l} with $j=1$. 
Arguing inductively as done in the case $N=1$, we obtain \eqref{E_p(w)>Nc0-eps^l} for the general case $j\geq2$.
\end{proof}
\begin{rem}\label{rem:alg}
The estimate \eqref{E_p(w)>Nc0-eps^l} with the sequence $\{k_j\}_{j\in\mathbb{N}}$ defined in \eqref{eq:exp_alg} holds for any $m>1$ and $n\geq2$ in \eqref{eq:D-F}.
However, it is very important to distinguish the cases $n\leq m+2$ and $n>m+2$.
Indeed, if $n\leq m+2$ the increasing sequence in \eqref{eq:exp_alg} is unbounded because $\alpha\geq1$ 
and we can replace the estimate \eqref{E_p(w)>Nc0-eps^l} with
\begin{equation*}
	E_\e[w]\geq N\gamma-C_j\varepsilon^{k}, \qquad \qquad k\in\mathbb{N},
\end{equation*}
provided that
\begin{equation*}
	E_\e[w]\leq N\gamma+\varepsilon^{k}, \qquad \qquad k\in\mathbb{N}.
\end{equation*}
On the other hand, the condition $n>m+2$ implies that $\alpha\in(0,1)$ and, consequently, 
the increasing sequence defined in \eqref{eq:exp_alg} is bounded and satisfies
\begin{equation}\label{eq:limitalpha}
	\lim_{j\to+\infty}k_j=\frac{\alpha}{1-\alpha}=\frac{n+m+2}{n-m-2}=1+\frac{2m+4}{n-m-2}=:\beta.
\end{equation}
\end{rem}

Reasoning as in the proof of Proposition \ref{prop:lower_deg}, we obtain the following result in the case \eqref{eq:D-F-onezero}.
\begin{prop}\label{prop:lower_deg2}
Consider the functional \eqref{eq:energy} with $D, F$ given by \eqref{eq:D-F-onezero}.
Let $v:(a,b)\rightarrow\{-1,+1\}$ a piecewise constant function with exactly $N$ discontinuities (as in \eqref{eq:vstruct}) and define the sequence
\begin{equation}\label{eq:exp_alg2}
	k_1=0,\qquad \qquad k_2:=\frac{n+1}{2n}, \qquad\qquad  k_{j+1}:=\frac{n+2}{2n}(k_{j}+1), \qquad j\geq2.
\end{equation}
Then, for any $j\in \mathbb N$ there exist constants $\delta_j>0$ and $C_j>0$ such that if $w\in H^1$ satisfies
\begin{equation*}
	\|w-v\|_{L^1}\leq\delta_j,
\end{equation*}
and
\begin{equation*}
	E_\e[w]\leq N\gamma+\varepsilon^{k_{j}},
\end{equation*}
with $\varepsilon$ sufficiently small, then
\begin{equation*}
	E_\e[w]\geq N\gamma-C_j\varepsilon^{k_{j+1}},
\end{equation*}
where $\gamma$ is defined in \eqref{eq:gamma}.
\end{prop}

\begin{proof}
The proof is exactly the same of the one of Proposition \ref{prop:lower_deg} (we left the same notation), 
with the only difference that the crucial estimate \eqref{stima-l=1} when $D, F$ are given by \eqref{eq:D-F-onezero} becomes
\begin{equation*}
	\begin{aligned}
		E_\e[w]&\geq\int_{x_1}^{y_1}\left[\frac\e2D(w)w_x^2+\frac{F(w)}{\e}\right]\,dx
		\geq\int_{w(x_1)}^{w(y_1)} \sqrt{2D(s)F(s)}\,ds\\
		&\geq\gamma-\int_{1-C_j\varepsilon^\frac1{n}}^1 \sqrt{2D(s)F(s)}\,ds-\int_{-1}^{-1+C_j\varepsilon^\frac1{n}}\sqrt{2D(s)F(s)}\,ds\\
		&\geq\gamma- C_j(\e^\alpha+\e^{\frac{n+2}{2n}})\geq\gamma-C_j\e^{\frac{n+2}{2n}},	
	\end{aligned}
\end{equation*}
where $\alpha$ is defined in \eqref{eq:exp_alg} and we used $\alpha\geq\frac{n+2}{2n}$, $\e\ll1$.
\end{proof}

Observe that the sequence $\{k_j\}_{j\in\mathbb N}$ in Proposition \ref{prop:lower_deg2} is exactly \eqref{eq:exp_alg} with $m=0$.
Hence, the same considerations of Remark \ref{rem:alg} hold true also in the case \eqref{eq:D-F-onezero}, 
and in the case $n>2$, the sequence \eqref{eq:exp_alg2} is increasing, bounded and 
\begin{equation*}\label{}
	\lim_{j\to+\infty}k_j=\frac{n+2}{n-2}=1+\frac{4}{n-2},
\end{equation*}
that is \eqref{eq:limitalpha} with $m=0$.

We conclude this section by showing an example of sequence $u^\e$ converging in $L^1(a,b)$ to a piecewise constant function $v$ as in \eqref{eq:vstruct} 
and for what the equality in the limit \eqref{eq:liminf-E} holds true.
\begin{prop}\label{prop:example}
Let $D$ and $F$ be as in \eqref{eq:D-F} or \eqref{eq:D-F-onezero} and fix $v$ as in \eqref{eq:vstruct}. 
There exists a function $u^\e$ satisfying 
\begin{equation}\label{eq:E=Ngamma}
		\lim_{\e\to0^+}\|u^\e-v\|_{{}_{L^1}}=0, \qquad \mbox{ and } \qquad \lim_{\e\to0^+}E_\e[u^\e]=N\gamma,
\end{equation} 
where $\gamma$ is defined in \eqref{eq:gamma}.
\end{prop}

\begin{proof}
We have already introduced a function satisfying \eqref{eq:E=Ngamma} in Section \ref{sec:compactons}:
the proof consists in checking that the function $\varphi_N^\e$ defined in \eqref{eq:compactons-}
with standing wave $\Phi$ satisfying \eqref{eq:ST} with $\Phi(0)=0$, whose existence has been proved in Proposition \ref{prop:standing} in the case of $D,F$ given by \eqref{eq:D-F}, and in Proposition \ref{prop:standing-onezero} when $D, F$ are given by \eqref{eq:D-F-onezero}, satisfies \eqref{eq:E=Ngamma} for any $N\in\mathbb N$.

The convergence in $L^1$ is a consequence of the behavior of the standing wave $\Phi$ as $\e\to0^+$.
Let us consider the case \eqref{eq:D-F}, being the other case very similar.
Indeed, in the case $n<m+2$ we have $\Phi(\pm\omega_\e)=\pm1$ and $\omega_\e$ defined in \eqref{eq:omega} goes to $0$ as $\e\to0^+$;
on the other hand, if $n\geq m+2$ we can write $\Phi(x)=\Psi(x/\e)$, where $\displaystyle\lim_{s\to\pm\infty}\Psi(s)=\pm1$, see the estimates \eqref{eq:exp-decay}-\eqref{eq:alg-decay}.

Now, let us check that $\varphi_N^\e$ satisfies the second limit in \eqref{eq:E=Ngamma}. 
First, consider the case $n\geq m+2$: by definitions of energy \eqref{eq:energy}, function $\varphi_N^\e$ \eqref{eq:compactons-},
and the fact that $\Phi$ satisfies \eqref{eq:Cauchy-ST}, we get
\begin{align*}
	E_\e[\varphi_N^\e]&=\sum_{i=1}^{N}\int_{m_i}^{m_{i+1}}\left[\frac{\e}{2}D(\Phi)(\Phi')^2+\frac{F(\Phi)}{\e}\right]\,dx
	=\sum_{i=1}^{N}(-1)^i\int_{l_j}^{l_{j+1}}\left[\sqrt{2D(\Phi)F(\Phi)}\Phi'\right]\,dx\\
	&=\sum_{i=1}^{N}(-1)^i\int_{\Phi(l_i)}^{\Phi(l_{i+1})}\sqrt{2D(s)F(s)}\,ds,
\end{align*}
where $l_i:=(-1)^i(m_i-h_i)$, for $i=1,\dots,N$ and $l_{N+1}=(-1)^{N}(b-h_N)$.
Since $|\Phi|<1$ if $n\geq m+2$ both in the case \eqref{eq:D-F} and \eqref{eq:D-F-onezero},
 we have $E_\e[\varphi_N^\e]<N\gamma$.
Using such an upper bound and the lower bounds obtained in Propositions \ref{prop:lower}-\ref{prop:lower2}-\ref{prop:lower_deg},
we obtain \eqref{eq:E=Ngamma} in the case $n\geq m+2$.
Proceeding in the same way and taking advantage of the fact that $\varphi_N^\e(h_i\pm\omega_\e)=\pm(-1)^{i+1}$, 
we end up with $E_\e[\varphi_N^\e]=N\gamma$ in the case $D,F$ given by \eqref{eq:D-F} with $n<m+2$.
Similarly, we obtain $E_\e[\varphi_N^\e]<N\gamma$ in the case \eqref{eq:D-F-onezero} with $n<m+2$,
and we can conclude that $\varphi_N^\e$ satisfies \eqref{eq:E=Ngamma} in both the cases \eqref{eq:D-F}-\eqref{eq:D-F-onezero} for any $m>1$ and $n\geq2$.
\end{proof}

\section{Slow motion of the solutions}\label{sec:slow}
In this final section, we present the main results of the paper, regarding the slow evolution of solutions to \eqref{eq:model}-\eqref{eq:Neu},
when $D$ and $F$ are given by \eqref{eq:D-F} or \eqref{eq:D-F-onezero}.
To be more precise, when $D$ and $F$ are given by \eqref{eq:D-F} we focus the attention on the case $n\geq m+2$.
As we saw in Section \ref{sec:stationary}, in this framework the only stationary solutions to \eqref{eq:model}-\eqref{eq:Neu} are constants 
($\pm1$ stable, $0$ unstable) or periodic;
however, we shall prove persistence of metastable patterns with a non-stationary $N$-\emph{transition layer structure} 
for (at least) an exponentially long time, i.e. for a time $T_\e\geq\nu\exp(c/\e)$ as $\e\to0^+$ when $n=m+2$ (taking advantage of Proposition \ref{prop:lower}),
and for an algebraically long time $T_\e\geq\nu\e^{-\beta}$ as $\e\to0^+$ when $n>m+2$ (taking advantage of Proposition \ref{prop:lower_deg}).
Moreover, we consider \eqref{eq:D-F-onezero} with $n=2$ and $m\in(1,2]$ or $n>2$, $m>1$:
in the first case, thanks to Proposition \ref{prop:lower2}, we can prove persistence of metastable patterns for an exponentially long time, 
while in the second one, thanks to Proposition \ref{prop:lower_deg2}, we can prove that some unstable structures persist for an algebraically long time. 

In order to obtain all the aforementioned results, we used the energy approach introduced by Bronsard and Kohn \cite{BrKo90} 
to study slow motion of solutions to the classical Allen--Cahn equation, which is formally obtained by choosing $m=0$ and $n=2$ in \eqref{eq:D-F}.
In particular, in \cite{BrKo90} the authors prove that some solutions maintain the same $N$-transition layer structure of the initial datum
(at least) for a time of $T_\e\geq\nu\e^{-k}$ as $\e\to0^+$, for any $k\in\mathbb N$.
The energy approach of \cite{BrKo90} is quite elementary yet powerful and it can be adapted for many different evolution PDEs:
without claiming to be complete, we list some references where the energy approach \cite{BrKo90} has been used to study slow motion of solutions.
In \cite{Grnt95}, Grant improves the energy approach of \cite{BrKo90} to obtain exponentially slow motion for the Cahn--Morral system;
then, this improved energy approach has been applied to both hyperbolic PDEs like in \cite{FLM19} and in parabolic PDEs with nonlinear diffusion \cite{FHLP20,FPS21}.
After establishing energy estimates like \eqref{eq:energyestimates} and the crucial lower bound \eqref{eq:lower-theta} 
obtained in the Propositions \ref{prop:lower}, \ref{prop:lower2}, \ref{prop:lower_deg}, \ref{prop:lower_deg2}, the energy approach is standard
and the procedure to conclude the proof of the slow motion of solutions is very similar to the previous works, 
see among others the aforementioned papers \cite{BrKo90,FHLP20,FLM19,FPS21,Grnt95}.
Then, the main novelty in the proof of our main results are the energy estimates \eqref{eq:energyestimates} and the variational results contained in Section \ref{sec:lower}.
For the sake of completeness, we show how to proceed and obtain slow motion results for solutions to \eqref{eq:model}-\eqref{eq:Neu}
satisfying the energy estimates \eqref{eq:energyestimates} once one has a lower bound on the energy of the type \eqref{eq:lower-theta}.
With this aim, we rewrite all the Propositions \ref{prop:lower}, \ref{prop:lower2}, \ref{prop:lower_deg}, \ref{prop:lower_deg2} in a single one as follows.

\begin{prop}\label{prop:alltogether}
Consider the functional \eqref{eq:energy} with $D, F$ given by \eqref{eq:D-F} or \eqref{eq:D-F-onezero}.
Fix $v$ as in \eqref{eq:vstruct} and $r>0$ such that \eqref{eq:r} holds true.
Then, there exist $\theta:(0,\infty)\to(0,\infty)$ converging to $0$ as $\e\to0^+$, $\e_0,C,\bar\delta>0$ such that if $u\in H^1(a,b)$ satisfies 
\begin{equation}\label{eq:u-v-bar}
	\|u-v\|_{{}_{L^1}}\leq\bar\delta,
\end{equation}
and
\begin{equation}\label{eq:add-ass}
	E_\e[u]\leq N\gamma+\theta(\e), \qquad \qquad \mbox{ for } \, \e\in(0,\e_0),
\end{equation}
then the energy functional \eqref{eq:energy} satisfies
\begin{equation*}
	E_\e[u]\geq N\gamma-C\theta(\e), \qquad \qquad \mbox{ for any } \,\, \e\in(0,\e_0),
\end{equation*}
where the positive constant $\gamma$ is defined in \eqref{eq:gamma}.
\end{prop}
Notice that the additional assumption \eqref{eq:add-ass} is needed only in Propositions \ref{prop:lower_deg}, \ref{prop:lower_deg2},
but, as we will see, in our applications \eqref{eq:add-ass} is always satisfied, 
so we include it in order to treat contemporarily all the results of Propositions \ref{prop:lower}, \ref{prop:lower2}, \ref{prop:lower_deg}, \ref{prop:lower_deg2}.
We remember that the explicit expression of the function $\theta$ in all the cases we want to treat.
\begin{description}
\item[(E1)] If $n=m+2$ and $m>2$ in \eqref{eq:D-F}, from \eqref{eq:lower} it follows that
$$\theta(\e)=\exp\left\{-A\sqrt{n}/\varepsilon\right\}, \qquad A\in(0,4r).$$ 
\item[(E2)] If $n=m+2$ and $m\in(1,2]$ in \eqref{eq:D-F}, \eqref{eq:lower2} implies
$$\theta(\e)=\exp\left\{-A/\varepsilon\right\}, \qquad A\in(0,2r).$$
\item[(E3)] If $n=2$ and $m\in(1,2]$ in \eqref{eq:D-F-onezero}, \eqref{eq:1-deg-cm} gives
$$\theta(\e)=\exp\left\{-A/\varepsilon\right\}, \qquad A\in(0,2^\frac{2-m}2r).$$
\item[(E4)] If $n>m+2$ in \eqref{eq:D-F}, \eqref{E_p(w)>Nc0-eps^l} yields 
$$\theta(\e)=\e^{k_{j+1}}, \qquad  \mbox{ with } \quad \{k_j\}_{j\in\mathbb N}\, \mbox{ defined in \eqref{eq:exp_alg}}.$$
\item[(E5)] If $n>2$ in \eqref{eq:D-F-onezero}, then Proposition \ref{prop:lower_deg2} and Remark \ref{rem:alg} imply
$$\theta(\e)=\e^{k_{j+1}}, \qquad  \mbox{ with } \quad \{k_j\}_{j\in\mathbb N}\, \mbox{ defined in \eqref{eq:exp_alg2}}.$$
\end{description}

Thanks to Proposition \ref{prop:alltogether} and the energy estimates \eqref{eq:energyestimates}, we shall prove that 
solutions to \eqref{eq:model}-\eqref{eq:Neu} starting with initial data which have a $N$-transition layer structure,
maintain such a structure for a time $T_\e\geq\nu\theta(\e)^{-1}$, for any $\nu>0$.
Since $\theta$ is given by \textbf{(E1)}-\textbf{(E5)}, we have that $T_\e\to\infty$ as $\e\to0^+$ and,
precisely, in \textbf{(E1)}-\textbf{(E2)}-\textbf{(E3)} $T_\e$ is (at least) exponentially large for $\e\ll1$ while
it is (at least) algebraically large in \textbf{(E4)}-\textbf{(E5)}.

Before stating and proving our main result, let us give the definition of a function with a $N$-transition layer structure,
which is a family of functions $\{u^\e\}_{\e>0}$ converging in $L^1$ to some piecewise constant function $v$ as in \eqref{eq:vstruct} as $\e\to0^+$,
and such that the energy $E_\e[u^\e]$ exceeds of a small quantity $\theta$ the minimum energy to have $N$ transition between $-1$ and $+1$. 

\begin{defn}\label{def:TLS}
Let $D$, $F$ as in \eqref{eq:D-F} or \eqref{eq:D-F-onezero}, with corresponding $\theta$ as in \textbf{(E1)}-\textbf{(E5)}
and let $v$ be a piecewise constant function like in \eqref{eq:vstruct}.
We say that a family of functions $\{u^\e\}\in H^1(a,b)$ has an \emph{$N$-transition layer structure} if 
\begin{equation}\label{eq:ass-u0}
	\lim_{\varepsilon\rightarrow 0} \|u^\varepsilon-v\|_{{}_{L^1}}=0,
\end{equation}
and there exists $C>0$ such that
\begin{equation}\label{eq:energy-ini}
	E_\varepsilon[u^\varepsilon]\leq N\gamma+C\theta(\e),
\end{equation}
for any $\varepsilon\ll1$, where the energy $E_\e$ is defined \eqref{eq:energy} and the constant $\gamma$ is given by \eqref{eq:gamma}.
\end{defn}

As we have already mentioned, the main goal of this section is to show that if the initial datum $u_0^\e$ has an $N$-\emph{transition layer structure}, 
then the solution to the IBVP \eqref{eq:model}-\eqref{eq:Neu}-\eqref{eq:initial} maintains such a structure for a time $T_\e\geq\nu\theta(\e)^{-1}$ as $\e\to0^+$, for any $\nu>0$.
We underline that the condition \eqref{eq:ass-u0} fixes the number of transitions between $-1$ and $+1$ and their relative positions as $\e\to0^+$,
while the condition \eqref{eq:energy-ini} requires that $u^\e$ makes these transitions in an ``energetical efficient" way.
Observe that, condition \eqref{eq:energy-ini} is automatically satisfied for all the time by \eqref{eq:energyestimates}.
Our main result ensures that condition \eqref{eq:ass-u0} is satisfied for a time $T_\e\geq\nu\theta(\e)^{-1}$ as $\e\to0^+$.

\begin{thm}\label{thm:main}
Assume that $D$, $F$ are given by \eqref{eq:D-F} or \eqref{eq:D-F-onezero}.
Let $v,r$ be as in \eqref{eq:vstruct}-\eqref{eq:r} and $\theta$ as in \emph{\textbf{(E1)}-\textbf{(E5)}}.
If $u^\varepsilon$ is the solution of \eqref{eq:model}-\eqref{eq:Neu}-\eqref{eq:initial} 
with initial datum $u_0^{\varepsilon}$ satisfying \eqref{eq:ass-u0} and \eqref{eq:energy-ini}, then, 
\begin{equation}\label{eq:limit}
	\sup_{0\leq t\leq \nu\theta(\varepsilon)^{-1}}\|u^\varepsilon(\cdot,t)-v\|_{{}_{L^1}}\xrightarrow[\varepsilon\rightarrow0]{}0,
\end{equation}
for any $\nu>0$.
\end{thm}

The proof of \eqref{eq:limit} is a consequence of the following result, which is obtained by using \eqref{eq:energyestimates} and Proposition \ref{prop:alltogether}.

\begin{prop}\label{prop:L2-norm}
Assume that $D$, $F$ are given by \eqref{eq:D-F} or \eqref{eq:D-F-onezero},
and consider the solution $u^\e$ to \eqref{eq:model}-\eqref{eq:Neu}-\eqref{eq:initial} 
with initial datum $u_0^{\varepsilon}$ satisfying \eqref{eq:ass-u0} and \eqref{eq:energy-ini}.
Then, there exist positive constants $\varepsilon_0, C_1, C_2>0$ (independent on $\varepsilon$) such that
\begin{equation}\label{L2-norm}
	\int_0^{C_1\varepsilon^{-1}\theta(\e)^{-1}}\|u_t^\varepsilon\|^2_{{}_{L^2}}dt\leq C_2\theta(\varepsilon)\e,
\end{equation}
for all $\varepsilon\in(0,\varepsilon_0)$.
\end{prop}

\begin{proof}
By using the assumptions \eqref{eq:ass-u0} and \eqref{eq:energy-ini}, choose $\varepsilon_0>0$ so small that for all $\varepsilon\in(0,\varepsilon_0)$, we have
\begin{equation}\label{1/2delta}
	\|u_0^\varepsilon-v\|_{{}_{L^1}}\leq\frac12\bar\delta,
\end{equation}
where $\bar\delta$ is the constant appearing in condition \eqref{eq:u-v-bar}, and the estimates \eqref{eq:energy-ini} holds.
We claim that if
\begin{equation}\label{claim1}
	\int_0^{\hat T}\|u_t^\varepsilon\|_{{}_{L^1}}dt\leq\frac12\bar\delta,
\end{equation}
for some $\hat T>0$, then there exists $C>0$ such that
\begin{equation}\label{claim2}
	E_\varepsilon[u^\varepsilon](\hat T)\geq N\gamma-C\theta(\varepsilon).
\end{equation}
Indeed, inequality \eqref{claim2} follows from Proposition \ref{prop:alltogether} if $u^\varepsilon(\cdot,\hat T)$ satisfies \eqref{eq:u-v-bar} and \eqref{eq:add-ass}.
The latter condition holds true because $E_\varepsilon[u^\varepsilon](\hat T)\leq E_\varepsilon[u^\varepsilon](0)\leq N\gamma+C\theta(\e)$, 
where we used the energy estimates \eqref{eq:energyestimates} and the assumption \eqref{eq:energy-ini}.
As concerning the condition \eqref{eq:u-v-bar} at time $\hat T$, by using triangle inequality, \eqref{1/2delta} and \eqref{claim1}, we obtain
\begin{equation*}
	\|u^\varepsilon(\cdot,\hat T)-v\|_{{}_{L^1}}\leq\|u^\varepsilon(\cdot,\hat T)-u_0^\varepsilon\|_{{}_{L^1}}+\|u_0^\varepsilon-v\|_{{}_{L^1}}
	\leq\int_0^{\hat T}\|u_t^\varepsilon\|_{{}_{L^1}}+\frac12\bar\delta\leq\bar\delta,
\end{equation*}
and we proved the claim.
Next, by using the inequalities \eqref{eq:energyestimates}, \eqref{eq:energy-ini} and \eqref{claim2}, we infer
\begin{equation}\label{L2-norm-Teps}
	\int_0^{\hat T}\|u_t^\varepsilon\|^2_{{}_{L^2}}dt=\e\left(E_\e[u_0^\e]-E_\e[u_\e](\hat T)\right)\leq C_2\theta(\varepsilon)\e.
\end{equation}
It remains to prove that inequality \eqref{claim1} holds for $\hat T\geq C_1\e^{-1}\theta(\varepsilon)^{-1}$.
If 
\begin{equation*}
	\int_0^{+\infty}\|u_t^\varepsilon\|_{{}_{L^1}}dt\leq\frac12\bar\delta,
\end{equation*}
there is nothing to prove. 
Otherwise, choose $\hat T$ such that
\begin{equation*}
	\int_0^{\hat T}\|u_t^\varepsilon\|_{{}_{L^1}}dt=\frac12\bar\delta.
\end{equation*}
Using H\"older's inequality and \eqref{L2-norm-Teps}, we deduce
\begin{equation*}
	\frac12\bar\delta\leq[\hat T(b-a)]^{1/2}\biggl(\int_0^{\hat T}\|u_t^\varepsilon\|^2_{{}_{L^2}}dt\biggr)^{1/2}\leq
	\left[\hat T(b-a)C_2\varepsilon\theta(\varepsilon)\right]^{1/2}.
\end{equation*}
It follows that there exists $C_1>0$ such that
\begin{equation*}
	\hat T\geq C_1\varepsilon^{-1}\theta(\varepsilon)^{-1},
\end{equation*}
and the proof is complete.
\end{proof}

Now, we have all the tools to prove \eqref{eq:limit}.
\begin{proof}[Proof of Theorem \ref{thm:main}]
Triangle inequality gives
\begin{equation}\label{trianglebar}
	\|u^\varepsilon(\cdot,t)-v\|_{{}_{L^1}}\leq\|u^\varepsilon(\cdot,t)-u_0^\varepsilon\|_{{}_{L^1}}+\|u_0^\varepsilon-v\|_{{}_{L^1}},
\end{equation}
for all $t\in[0,\nu\theta(\varepsilon)^{-1}]$. 
The term $\|u_0^\varepsilon-v\|_{{}_{L^1}}$ tends to $0$ by assumption \eqref{eq:ass-u0}.
Regarding the first term in the right hand side of \eqref{trianglebar}, take $\varepsilon$ so small that $C_1\varepsilon^{-1}\geq\nu$;
thus we can apply Proposition \ref{prop:L2-norm} and by using H\"older's inequality and \eqref{L2-norm}, we infer
\begin{equation*}
	\sup_{0\leq t\leq \nu\theta(\varepsilon)^{-1}}\|u^\e(\cdot,t)-u^\e_0\|_{{}_{L^1}}\leq\int_0^{\nu\theta(\varepsilon)^{-1}}\|u_t^\e(\cdot,t)\|_{{}_{L^1}}\,dt\leq C\sqrt\e,
\end{equation*}		
for all $t\in[0,\nu\theta(\varepsilon)^{-1}]$. Hence, \eqref{eq:limit} follows.
\end{proof}

\subsection{Layer dynamics}
In this section we apply a standard procedure (cfr. \cite{FHLP20,FLM19,FPS21,Grnt95}),
which allows us to obtain an upper bound on the velocity of the transition points, once \eqref{eq:limit} has been established.

Fix $v$ as in \eqref{eq:vstruct}, $r$ as in \eqref{eq:r} and define its {\it interface} $I[v]$ as
\begin{equation*}
	I[v]:=\{h_1,h_2,\ldots,h_N\}.
\end{equation*}
For an arbitrary function $u:[a,b]\rightarrow I$ and an arbitrary closed subset $K\subset\R\backslash\{-1,1\}$,
the {\it interface} $I_K[u]$ is defined by
\begin{equation*}
	I_K[u]:=u^{-1}(K).
\end{equation*}
Finally, we recall that for any $X,Y\subset\mathbb{R}$ the {\it Hausdorff distance} $d(X,Y)$ between $X$ and $Y$ is defined by 
\begin{equation*}
	d(X,Y):=\max\biggl\{\sup_{x\in X}d(x,Y),\,\sup_{y\in Y}d(y,X)\biggr\},
\end{equation*}
where $d(x,Y):=\inf\{|y-x|: y\in Y\}$. 

First, we prove a purely variational in character result, which states that, if a function $u\in H^1(a,b)$ is close to $v$ in $L^1$ 
and $E_\varepsilon[u]$ exceeds of a sufficiently small quantity the minimum energy to have $N$ transitions, then 
the distance between the interfaces $I_K[u]$ and $I_K[v]$ is small.  
\begin{lem}\label{lem:interface}
Let $D$ and $F$ be as in \eqref{eq:D-F} or \eqref{eq:D-F-onezero}.
Given $\delta_1\in(0,r)$ and a closed subset $K\subset \R\backslash\{-1,1\}$, 
there exist positive constants $\hat\delta,\varepsilon_0$ (independent on $\e$) and $M>0$ such that for any $u\in H^1(a,b)$ satisfying 
\begin{equation}\label{eq:lem-interf}
	\|u-v\|_{{}_{L^1}}<\hat\delta \qquad \quad \mbox{ and } \qquad \quad E_\varepsilon[u]\leq N\gamma+M,
\end{equation}
for all $\varepsilon\in(0,\varepsilon_0)$, we have
\begin{equation}\label{lem:d-interfaces}
	d(I_K[u], I[v])<\tfrac12\delta_1.
\end{equation}
\end{lem}
\begin{proof}
Fix $\delta_1\in(0,r)$ and choose $\rho>0$ small enough that 
\begin{equation*}
	I_\rho:=(-1-\rho,-1+\rho)\cup(1-\rho,1+\rho)\subset\R\backslash K, 
\end{equation*}
and 
\begin{equation*}
	\inf\left\{\left|\int_{\xi_1}^{\xi_2}\sqrt{2D(s)F(s)}\,ds\right| : \xi_1\in K, \xi_2\in I_\rho\right\}>2M,
\end{equation*}
where
\begin{equation*}
	M:=2N\max\left\{\int_{-1}^{-1+\rho}\sqrt{2D(s)F(s)}\,ds, \, \int_{1-\rho}^{1}\sqrt{2D(s)D(s)}\,ds \right\}.
\end{equation*}
By reasoning as in the proof of \eqref{2points} in Proposition \ref{prop:lower}, we can prove that for each $i$ there exist
\begin{equation*}
	x^-_{i}\in(h_i-\delta_1/2,h_i) \qquad \textrm{and} \qquad x^+_{i}\in(h_i,h_i+\delta_1/2),
\end{equation*}
such that
\begin{equation*}
	|u(x^-_{i})-v(x^-_{i})|<\rho \qquad \textrm{and} \qquad |u(x^+_{i})-v(x^+_{i})|<\rho.
\end{equation*}
Suppose that \eqref{lem:d-interfaces} is violated.
Hence, $d(I_K[u], I[v])\geq\frac12\delta_1$, meaning that there exists (at least) a point $\zeta$ such that
$u(\zeta)=\xi_1\in K$ and $\displaystyle\min_{i=1,\dots,N}|h_i-\zeta|\geq\tfrac12\delta_1$.
In particular, notice that $\zeta\notin[x^-_{i},x^+_{i}]$.
Using \eqref{eq:ineq}, we deduce
\begin{align}
	E_\varepsilon[u]\geq&\sum_{i=1}^N\left|\int_{u(x^-_{i})}^{u(x^+_{i})}\sqrt{2D(s)F(s)}\,ds\right|\notag\\ 
	& \qquad +\inf\left\{\left|\int_{\xi_1}^{\xi_2}\sqrt{2D(s)F(s)}\,ds\right| : \xi_1\in K, \xi_2\in I_\rho\right\}. \label{diseq:E1}
\end{align}
On the other hand, we have
\begin{align*}
	\left|\int_{u(x^-_{i})}^{u(x^+_{i})}\sqrt{2D(s)F(s)}\,ds\right|&\geq\int_{-1}^{1}\sqrt{2D(s)F(s)}\,ds\\
	&\qquad-\int_{-1}^{-1+\rho}\sqrt{2D(s)F(s)}\,ds\\
	&\qquad -\int_{1-\rho}^{1}\sqrt{2D(s)F(s)}\,ds\\
	&\geq\gamma-\frac{M}{N},
\end{align*}
where $\gamma$ is defined in \eqref{eq:gamma}.
Substituting the latter bound in \eqref{diseq:E1}, we deduce
\begin{equation*}
	E_\varepsilon[u]\geq N\gamma-M+\inf\left\{\left|\int_{\xi_1}^{\xi_2}\sqrt{2D(s)F(s)}\,ds\right| : \xi_1\in K, \xi_2\in I_\rho\right\}.
\end{equation*}
For the choice of $\rho$,  we obtain	
\begin{align*}
	E_\varepsilon[u]>N\gamma+M,
\end{align*}
which is a contradiction with assumption \eqref{eq:lem-interf}. 
Hence, the bound \eqref{lem:d-interfaces} is true.
\end{proof}

Thanks to Theorem \ref{thm:main} and Lemma \ref{lem:interface} we can prove the following result, 
which states that the velocity of the transition points is of $\mathcal{O}(\theta(\e))$.
\begin{thm}\label{thm:interface}
Let $D$ and $F$ be as in \eqref{eq:D-F} or \eqref{eq:D-F-onezero} and $\theta$ as in \emph{\textbf{(E1})-\textbf{(E5)}}.
Let $u^\varepsilon$ be the solution of \eqref{eq:model}-\eqref{eq:Neu}-\eqref{eq:initial}, 
with initial datum $u_0^{\varepsilon}$ satisfying \eqref{eq:ass-u0} and \eqref{eq:energy-ini}. 
Given $\delta_1\in(0,r)$ and a closed subset $K\subset\R\backslash\{-1,1\}$, set
\begin{equation*}
	t_\varepsilon(\delta_1)=\inf\{t:\; d(I_K[u^\varepsilon(\cdot,t)],I_K[u_0^\varepsilon])>\delta_1\}.
\end{equation*}
There exists $\varepsilon_0>0$ such that if $\varepsilon\in(0,\varepsilon_0)$ then
\begin{equation*}
	t_\varepsilon(\delta_1)>\theta(\varepsilon)^{-1}.
\end{equation*}
\end{thm}
\begin{proof}
Let $\varepsilon_0>0$ so small that \eqref{eq:ass-u0}-\eqref{eq:energy-ini}
imply $u_0^\varepsilon$ satisfies \eqref{eq:lem-interf} for all $\varepsilon\in(0,\varepsilon_0)$.
From Lemma \ref{lem:interface} it follows that
\begin{equation}\label{interfaces-u0}
	d(I_K[u_0^\varepsilon], I[v])<\tfrac12\delta_1.
\end{equation}
Now, consider $u^\varepsilon(\cdot,t)$ for all $t\in(0,\theta(\varepsilon)^{-1}]$.
If $\e$ is sufficiently small, then assumption \eqref{eq:lem-interf} is satisfied thanks to \eqref{eq:limit} and because $E_\varepsilon[u^\varepsilon](t)$ is a non-increasing function of $t$, see \eqref{eq:energyestimates}.
Then,
\begin{equation}\label{interfaces-u}
	d(I_K[u^\varepsilon(t)], I[v])<\tfrac12\delta_1,
\end{equation}
for all $t\in(0,\theta(\varepsilon)^{-1}]$. 
Combining \eqref{interfaces-u0} and \eqref{interfaces-u}, we obtain
\begin{equation*}
	d(I_K[u^\varepsilon(t)],I_K[u_0^\varepsilon])<\delta_1,
\end{equation*}
for all $t\in(0,\theta(\varepsilon)^{-1}]$.
\end{proof}

\section{Conclusions}
\label{sec:conclusions}
Let us make some comments about Theorems \ref{thm:main}-\ref{thm:interface}.
First of all, in the proof of Proposition \ref{prop:example} we showed that the function $\varphi_N^\e$ defined in \eqref{eq:compactons-} converges in $L^1$ 
to a piecewise constant function $v$ as in \eqref{eq:vstruct} and satisfies $E_\e[\varphi_N^\e]\leq N\gamma$, 
in both the cases \eqref{eq:D-F} and \eqref{eq:D-F-onezero} for any $m\geq0$ and $n\geq2$.
Then, we can say that $\varphi_N^\e$ has a $N$-transition layer structure
and the solution to \eqref{eq:model}-\eqref{eq:Neu}-\eqref{eq:initial} with $u_0^\e=\varphi_N^\e$ maintains
such a transition layer structure for a time $T_\e\geq\nu\theta(\e)^{-1}$.

This result is very surprising in the case \eqref{eq:D-F} with $n\geq m+2$ because $\varphi_N^\e$ is far from being a stationary solution.
Nevertheless, thanks to Theorems \ref{thm:main}-\ref{thm:interface} we can state that
in the case $n=m+2$, since \textbf{(E1)}-\textbf{(E2)} give $T_\e\geq\nu\exp\left\{A\sqrt{n}/\varepsilon\right\}$ for $m\geq2$ 
and $T_\e\geq\nu\exp\left\{A/\varepsilon\right\}$ for $m\in(1,2]$, the solution exhibits \emph{exponentially slow motion}:
it maintains the same transition layer structure for (at least) an exponentially long time and the layers move with an exponentially small speed.
Therefore, we generalized the classical results on the Allen--Cahn equation,
which is \eqref{eq:model}-\eqref{eq:D-F} with $m=0$ and $n=2$, see among others \cite{CaPe89}, \cite{ChenX04}.
Notice that the positive constant $A$ in \textbf{(E1)}-\textbf{(E2)} is strictly related to the constant $r$ in \eqref{eq:r}
and so, to the distance between the two closest transition points: the further the points are, the slower the evolution of the solution.
Moreover, comparing \textbf{(E1)} and \textbf{(E2)}, we see that the degree of the degeneracy in $D$ slows down the evolution of the solution,
in the sense that the larger $m$, the slower the solution evolves.
 
On the other hand, if $n>m+2$ we only proved an \emph{algebraic slow motion}, because the sequence $\{k_j\}_{j\in\mathbb N}$ in \textbf{(E4)}
is bounded and satisfies \eqref{eq:limitalpha}.
It is worthy to observe that the limit $\beta$ of the increasing sequence $\{k_j\}_{j\in\mathbb N}$ diverges as $n\to m+2$.
In this case, we generalized the results of \cite{BetSme2013}.

Regarding the case of a diffusion coefficient with only one degenerate point \eqref{eq:D-F-onezero},
thanks to Theorems \ref{thm:main}-\ref{thm:interface} and \textbf{(E3)},
we can say that some solutions exhibit exponentially slow motion if $n=2$ for any $m\in(1,2]$.
However, it is important to distinguish two cases.
As we proved in Proposition \ref{prop:stat-onezero}, the profile $\varphi_N^\e(-x)$ is a stationary solution if $N$ is even,
so we do not expect evolution at all.
Thanks to Theorems \ref{thm:main}-\ref{thm:interface} and \textbf{(E3)}, we can state that a perturbation of $\varphi_{2N}^\e(-x)$
satisfying \eqref{eq:ass-u0}-\eqref{eq:energy-ini} remains very close to $\varphi_{2N}^\e(-x)$ (at least) for an exponentially long time,
but this is only a partial result.
Indeed, it has to be complemented with a stability result on the profile $\varphi_{2N}^\e(-x)$ and we imagine two possible scenarios:
either the stationary profile is asymptotically stable and the solution maintains the transition layer structure for all time,
or it is unstable and the solution exhibits again the phenomenon of metastability, namely it appears to be stable for a long time
and then the profile undergoes a drastic change (annihilation of some transitions).
The stability/instability of such profile (as the one of the compactons introduced in Proposition \ref{prop:ex-compactons})
is an open problem, which needs further investigation.
On the contrary, if the number of transition is $2N+1$ (odd) or $u_0^\e(a)<0$, we have 
exponentially slow motion of a non-stationary $N$-transition layer structure. 

In the case $n>2$, Theorems \ref{thm:main}-\ref{thm:interface} and \textbf{(E5)} tell us that we have algebraic slow motion for any $m>0$ and 
it is very important to notice that in this case the degeneracy of $D$ does not play a role in the slow evolution of the solution, see \eqref{eq:exp_alg2}.
This is a consequence of the fact that $F$ is degenerate at $-1$ (in the sense that $F''(-1)=0$) and there is no \emph{counterbalance} of the diffusion,
which satisfies $D(-1)>0$.
In other words, it is easy to extend our results to the general case
\begin{equation}\label{eq:general-D-F}
	D(u)=|1-u|^{m_1}|1+u|^{m_2}, \qquad \qquad F(u)=|1-u|^{n_1}|1+u|^{n_2},
\end{equation}
and if $n_i> m_i+2$ for $i=1$ or $i=2$, our results show algebraic slow motion of the solutions.

Finally, it is interesting to observe how the slow motion results are strictly related to the behavior of the standing wave,
described in Propositions \ref{prop:standing} and \ref{prop:standing-onezero}.
In the case that the standing wave does not touch the equilibria $\pm1$ (in order to exclude compactons or other stationary solutions),
some solutions exhibit exponentially slow motion if and only if the standing wave converges exponentially fast to $\pm1$ as $x\to\pm\infty$,
i.e. $n_i=m_i+2$ for $i=1,2$ in \eqref{eq:general-D-F}.
Otherwise, if $n_1>m_1+2$ or $n_2>m_2+2$ the standing wave converges algebraically fast to $+1$ or $-1$ and 
one can only observe solutions exhibiting algebraically slow motion.

\bibliography{bibliography}

\bibliographystyle{newstyle}

\end{document}